\documentclass[fleqn,preprint,11pt]{elsarticle}

\makeatletter
\def\ps@pprintTitle{%
 \let\@oddhead\@empty
 \let\@evenhead\@empty
 \def\@oddfoot{\centerline{\thepage}}%
 \let\@evenfoot\@oddfoot}
\makeatother

\usepackage{geometry}  

\usepackage{latexsym,amsmath,amssymb,amsthm}
\usepackage[outdir=./]{epstopdf}
\usepackage{bm}
\usepackage{url}
\usepackage{color}
\usepackage{enumitem}

\newtheorem{theorem}{Theorem}
\newtheorem{lem}[theorem]{Lemma}

\newtheorem{pro}[theorem]{Proposition}
\newtheorem{rmk}[theorem]{Remark}
\newtheorem{definition}[theorem]{Definition}

\usepackage{tabularx}
\usepackage{algorithm}
\usepackage{algorithmic}

\newcommand{\bi}{\begin{itemize}}
\newcommand{\ei}{\end{itemize}}
\newcommand{\ben}{\begin{enumerate}}
\newcommand{\een}{\end{enumerate}}

\newcommand{\be}{\begin{equation}}
\newcommand{\ee}{\end{equation}}
\newcommand{\ba}{\begin{aligned}}
\newcommand{\ea}{\end{aligned}}
\newcommand{\bea}{\begin{eqnarray}}
\newcommand{\eea}{\end{eqnarray}}

\newcommand{\tbox}[1]{{\mbox{\tiny #1}}}
\newcommand{\mbf}[1]{{\bm #1}}           % requires bm package

\newcommand{\bigO}{{\mathcal O}}

\newcommand{\ZZ}{\mathbb{Z}}
\newcommand{\RR}{\mathbb{R}}
\newcommand{\bsigma}{\bm{\sigma}}
\newcommand{\bxi}{\bm{\xi}}

\newcommand\bx{\mathbf{x}}

\newcommand\cs{\mathcal{S}}

\newcommand\cU{\mathcal{U}}         % unit cell
\newcommand\ccU{\overline{\cU}}   % closure of unit cell
\newcommand\pcU{\partial\cU}       % unit cell bdry
\newcommand\cK{\mathcal{K}}

\newcommand{\defeq}{:=}

\def\x{{\bf x}}
\def\X{{\bf X}}
\def\bx{{\bf x}}
\def\y{{\bf y}}

\def\c{{\bf c}}
\def\d{{\bf d}}
\def\s{{\bf s}}
\def\t{{\bf t}}
\def\c{{\bf c}}

\def\u{{\bf u}}
\def\v{{\bf v}}
\def\f{{\bf f}}
\def\w{{\bf w}}
\def\e{{\bf e}}
\def\g{{\bf g}}

\def\r{{\bf r}}
\def\T{{\bf T}}
\def\F{{\bf F}}
\def\bR{{\bf R}}

\def\n{{\bf n}}

\def\beq{\begin{equation}}
\def\eeq{\end{equation}}

\newcommand{\nx}{\n^\x}                       % n_x   (Stokes case)
\newcommand{\tor}{{\mathbb{T}^2}}              % 2-torus

\newcommand{\bc}{\begin{center}}
\newcommand{\ec}{\end{center}}
\newcommand{\bfi}{\begin{figure}}
\newcommand{\efi}{\end{figure}}
\newcommand{\ca}[2]{\caption{#1 \label{#2}}}

\newcommand{\ig}[2]{\includegraphics[#1]{#2}}

\def\FI{Flatiron Institute, Simons Foundation, New York, NY 10010}
\def\UNC{The University of North Carolina at Chapel Hill, Chapel Hill, NC, 27599}

\def\papertitle{An integral equation method for the simulation of doubly-periodic suspensions of rigid bodies in a shearing viscous flow}

\begin{document}

\begin{frontmatter}

\title{\papertitle}
% page style

%\pagestyle{myheadings}
%\thispagestyle{plain}

% title, addresses, etc.
\author{Jun Wang\fnref{fi}}
\address[fi]{\FI}
\ead{jwang@flatironinstitute.org}
 
\author{Ehssan Nazockdast\fnref{unc}}
\address[unc]{\UNC}
\ead{ehssan@email.unc.edu}

\author{Alex Barnett\fnref{fi}}
%\address[fi]{\FI}
\ead{abarnett@flatironinstitute.org}

%\cortext[cor1]{Corresponding author.}

\begin{abstract}
  With rheology applications in mind,
  we present a fast solver for the
  time-dependent effective viscosity of
  an infinite lattice containing one or more neutrally buoyant smooth rigid particles per unit cell, in a two-dimensional Stokes fluid with given shear rate.
  At each time,
  the mobility problem %for particle velocities
  is reformulated as a
  %linear 
  2nd-kind boundary integral equation, %(BIE),
  then discretized to spectral
  accuracy by the Nystr\"om method and solved iteratively, giving typically
  10 digits of accuracy.
  Its solution controls the
  evolution of particle locations and angles in a first-order
  system of ordinary differential equations.
  The formulation is placed on a rigorous footing by defining
  a generalized periodic Green's function for the skew lattice.
  Numerically, the periodized integral operator is split into a near image
  sum---applied in linear time via the fast multipole method---%
  plus a correction field solved cheaply via proxy Stokeslets.
  We use barycentric quadratures to evaluate
  particle interactions and velocity fields accurately,
  even at distances much closer than the node spacing.
  Using first-order time-stepping we simulate, eg, 25 ellipses per unit cell
  to 3-digit accuracy on a desktop in 1 hour per shear time.
  Our examples show equilibration at long times, force chains,
  and two types of blow-ups (jamming) whose power laws match
  lubrication theory asymptotics.
\end{abstract}


\end{frontmatter}

\section{Introduction}
Flowing suspensions are ubiquitous in nature and industry, and have been 
used as model systems for theoretical studies of soft materials. 
They exhibit complex nonlinear rheological behavior, including shear-thinning, 
shear-thickening and existence of normal stress differences \cite{guazzelli2011physical}. 
 Previous theoretical \cite{brady1997microstructure, nazockdast2012microstructural}, 
computational \cite{sierou2001accelerated, wang2016spectral} and experimental studies 
\cite{gadala1980shear, cheng2011imaging, xu2014microstructure, gurnon2015microstructure}  
have established that particles near contact is the most likely configuration of particles in shear flows, 
and that near-contact hydrodynamic (lubrication) forces/stresses are key to 
determining the nonlinear rheology of suspensions, 
including their shear thickening behavior 
\cite{bender1998reversible, maranzano2001effects, jamali2019alternative}.
Accurate modeling of flow near such particles is also crucial to
understanding transport phenomena such as super-diffusion
\cite{souzy15}.

%-----------------------------------------------
Although previous
studies have largely focused on spheres,
many applications involve 
suspensions of non-spherical particles \cite[Ch.~5]{mewis2012colloidal}.
Moreover, experimental studies show that 
the rheology of concentrated suspensions, most notably their shear thickening behavior, 
are highly sensitive to their shape \cite{egres2005rheology, cwalina2016rheology}. 
For example, experiments on suspensions of cubic particles show that 
the viscosity diverges at volume fractions slightly below that of spherical particles, yet 
the divergence of viscosity and normal stress differences with volume fraction is stronger than those reported for spherical 
particles \cite{cwalina2016rheology}.

Currently, there is no numerical method that can accurately simulate lubrication forces and stresses in 
particulate suspensions of complex shape in the
Stokes (zero Reynolds number) and non-Brownian (infinite Peclet number)
regime.
The method that comes closest for \emph{spherical} suspensions is
 \emph{Stokesian Dynamics} (SD) \cite{durlofsky1987dynamic, brady1988stokesian, fossbrady,wang2016spectral}. 
 In that method, hydrodynamic interactions (HI)
 are divided into far-field and near-field.
The far-field mobility tensor is constructed through a multipole expansion truncated at the stresslet level;
the near-field resistance tensor is assumed to be pair-wise additive, and is computed by asymptotic lubrication solutions between two spheres. Despite its huge success, SD cannot be readily used to simulate non-spherical particles, especially their near-field HIs.

In this work, by contrast, we present a convergent numerical
method to solve the underlying Stokes mobility boundary value problem (BVP)
for arbitrary smooth particle shapes.
We use potential theory to reformulate this BVP
as a boundary integral equation (BIE)
for an unknown Stokeslet ``density'' function.
This BIE approach, which requires the suspending
fluid to be Newtonian,
has several advantages over conventional %volumetric
discretization (eg, finite elements):
i) since only the boundaries must be discretized
there is a large reduction in $N$, the number of unknowns;
ii) meshing becomes much simpler;
iii) unlike with volume discretization, well-conditioned formulations are possible;
and iv) close-touching surfaces may be handled without introducing a high density of small volume elements (as in \cite{haan98}).
A possible disadvantage is the loss of sparsity in the resulting linear
system; however, algorithms such as the fast multipole method (FMM)
\cite{lapFMM}
are available to apply the dense system matrix in $\bigO(N)$ time,
leading to optimal-complexity schemes.

Our BIE formulation is a doubly-periodic version of
the 2nd-kind formulation of
Karrila--Kim \cite{karrilakim}; also see Rachh--Greengard \cite{rachhgreengard}.
The motivation to address the periodic case is two-fold:
there is interest in
i) the rheology of {\em regular} arrays of rigid
particles, with prior numerical \cite{nunan84} and asymptotic
analysis work on discs in 2D and spheres in 3D \cite{berlyand09};
and
ii) the behavior of {\em random} suspensions
via numerical simulation of a
finite-sized representative volume element (RVE),
where periodic boundary conditions most closely model an infinite system
\cite{haan98,fossbrady}.
In both cases this is a {\em homogenization} problem, and the
(time-dependent) effective viscosity is sought.
Numerically, this demands fast application of a {\em periodized}
integral operator in an arbitrary skewing lattice,
for which we simplify and extend the recent
method of the last author, Marple, Veerapaneni, and Zhao \cite{ahb}.
This method (an extension of ideas of 
Larson--Higdon \cite{larsonhigdon2})
splits the periodized kernel into
a free-space Stokes FMM sum over nearby images, plus a correction
flow that solves an ``empty'' BVP with smooth data that is solved cheaply
using particular solutions.
This applies physical periodicity conditions,
avoids ad-hoc non-convergent lattice sums \cite{Krop04},
and, unlike particle-mesh Ewald (PME) methods which rely on FFTs
\cite{klint}, is spatially adaptive.

This work includes three other innovations.
1) We extend barycentric quadratures for spectrally accurate
near-boundary evaluation of the single-layer potential
due to the last author, Wu and Veerapaneni \cite{bwv},
to include the traction needed in our formulation.
This allows a spatial solve accurate to distances
as small as $d = \bigO(h^2)$, where $h$ is the on-surface node spacing
(see Figures~\ref{f:geom_stat_conv}--\ref{f:stat_conv}).
2) We present an efficient way to extract the effective viscosity
from the simulation via a line integral (see \eqref{mueff} and
section~\ref{s:mueff}).
3) We place the entire scheme on a rigorous footing,
proving that the periodized BIE is equivalent to the periodic mobility
BVP, by introducing a generalized periodic Stokes Green's function
(section~\ref{subsec:gper}).
As a byproduct we prove existence and uniqueness for the BVP itself
(Lemma~\ref{l:bvpunique} and Theorem~\ref{t:bvp}).

\begin{rmk}\label{r:contact}  % rrrrrrrrrrrrrrrrrrrrrrrrrrrrrrrrrrrrrrrrr
 We study the case of pure HIs, which is
  an accurate physical model for smooth particles not in contact.
Previous simulations of spheres in Stokes flow show that HIs alone 
are not sufficient for preventing particles from making contract and jamming at finite strains \cite{melrose1995pathological}.
For numerical tractability, non-hydrodynamic forces can be added, 
either through short-range repulsive forces \cite{sierou2001accelerated} 
or collision-free constraint methods \cite{lu2017contact, yan2019scalable},
when particles get closer than a cutoff distance.  
In experimental settings these near-contact interactions can, for example, be induced by 
the particles' surface roughness. In both experiments and simulations the presence of 
these non-hydrodynamic near-contact interactions 
break the fore-aft symmetry of linear Stokes equation leading to anisotropy in  
microstructure and nonlinear rheology \cite{gadala1980shear, brady1997microstructure}. 
In this study we do not include any non-hydrodynamic forces, yet, due to our
high accuracy and small time-steps we are able to runs simulations
for large times without particle collisions
when the volume fraction is moderate.
\end{rmk}   % rrrrrrrrrrrrrrrrrrrrrrrrrrrrrrrrrrrrrrrrrrrrrrrrrrrrrrrrrrr

The paper is organized as follows. In section~\ref{s:setup}, we set up the
problem geometry.
In section~\ref{s:bvp}, we describe the quasi-static BVP and review its mathematical properties.
In section~\ref{s:green}, we define a generalized periodic Stokes Green's function, prove its uniqueness
and existence in a constructive manner, and use it to generalize classical potential theory to our periodic case.
In section~\ref{s:bie}, we reformulate the BVP as a boundary integral equation
whose kernel is the periodic Green's function.
In section~\ref{sec:numerical}, we give a complete numerical algorithm for the quasi-static BVP, followed by a short discussion
of time stepping methods in section~\ref{s:step}. Numerical examples for both the quasi-static and time dependent problems
are given in section~\ref{s:num}.
We finish with conclusions in section~\ref{s:conc}.
Two appendices contain technical proofs needed for existence of the periodic
Green's function.

%fffffffffffffffffffff
\bfi  
\centering\ig{width=3in}{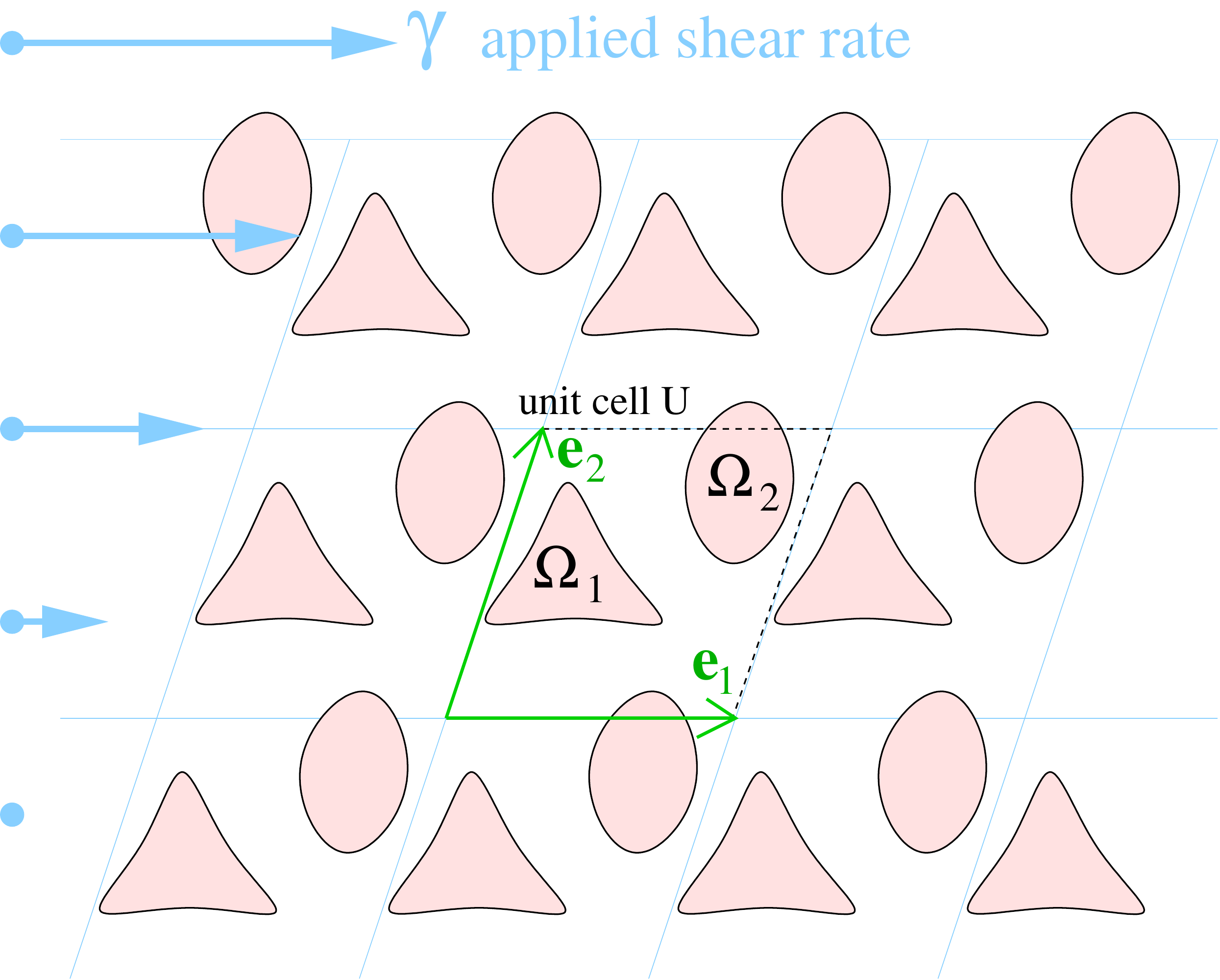}
\ca{Illustration of the periodic effective viscosity problem.
  A given shear rate $\gamma$ is applied to a
suspension in a viscous Newtonian background fluid
of an infinite lattice of one or more rigid, smooth,
neutrally buoyant particles per unit cell
(shown is the case of two per unit cell).
The suspending fluid has no-slip boundary conditions
on each particle. There are no Brownian effects.}{f:mobprob}
\efi

%%%%%%%%%%%%%%%%%%%%%%%%%%%%%%%%%%%%%%%%%%%%%%%%%%%%%%%%%%%%%%%%%%%%%%%%%%%%%%%
\section{Geometry and setup}
\label{s:setup}
We first need some geometry notation; see Figure~\ref{f:mobprob}.
At any fixed time $t$,
let $\e_1$ and $\e_2$ be two linearly independent vectors that define an infinite lattice in $\RR^2$. Fixing an arbitary center, they define a single unit cell $\cU$, with four walls comprising its boundary $\partial\cU = L\cup R\cup U\cup D$. We assume that, at this time $t$, in the unit cell $\cU$, there are a
collection of $N_o$ disjoint bounded smooth rigid objects $\{\Omega_j\}_{j=1}^{N_o}$, with boundaries $\Gamma_j=\partial \Omega_j$. Let $\Omega_{\Lambda}=\{\x\in\RR^2: \x+m_1\e_1+m_2\e_2\in \Omega_j, m_1,m_2\in \ZZ, j=1, \cdots, N_o\}$ be the infinite lattice of objects. The centroid of the object $\Omega_j$ will be $\x_j^c=\frac{1}{|\partial \Omega_j|}\int_{\partial \Omega_j} \x ds_{\x}$.
An applied shear will cause the lattice to change in time.
For convenience we fix $\e_1 = (1,0)$, but set $\e_2 = (\gamma t,1)$
for small times $t$ (see Section~\ref{s:step}), so that the shear rate is
the constant $\gamma$. \footnote{Often the shear rate is denoted by $\dot{\gamma}$; for simplicity we drop the dot.}
When explicit $t$-dependence is needed, we will write $\Omega_j(t)$, etc.

We further introduce some notations to describe the rigid body motion. Let $\x_j^{*}(0)$ be an arbitary point on $\Omega_j(0)$ that is distinct from $\x_j^c(0)$. They define a vector $\r_j(0)=\x_j^{*}(0)-\x_j^c(0)$.
At a later time $t$, they are located at $\x_j^{*}$ and $\x_j^c$, and the vector they define is $\r_j(t)=\x_j^{*}(t)-\x_j^c(t)$. We define the angular position $\theta_j(t)$ to be the angle from $\r_j(0)$ to $\r_j(t)$. Since a rigid body motion is assumed, $\theta_j(t)$ is independent of the choice of $\x_j^{*}$ and $\x_j^c$. The linear velocity is defined as $\v_j=\frac{d}{dt}\x_j^c(t)$ and the angular velocity is defined as $\omega_j=\frac{d}{dt}\theta_j(t)$.

For this problem, we have as input the initial configuration of the objects defined by $\{\Omega_j(0)\}_{j=1}^{N_o}$,
$\{\x_j^c(0)\}_{j=1}^{N_o}$ and $\{\theta_j(0)\}_{j=1}^{N_o}$, the unit cell defined by $\{\e_1(0), \e_2(0)\}$, 
and the background flow with shear rate $\gamma$ and viscosity $\mu$. We are interested in studying the evolution
of such a system. More specifically, we seek solutions of the following quantities: the velocity field of the
fluid $\u=\u(\x,t)$, the pressure field $p=p(\x,t)$, the motion of the objects
$\{\x_j^c(t),\theta_j(t)\}_{j=1}^{N_o}$, and macroscopic derived quantities such as the effective viscosity defined in \eqref{mueff}.

% SSSSSSSSSSSSSSSSSSSSSSSSSSSSSSSSSSSSSSSSSSSSSSSSSSSSSSSSSSSSSSSSSSSSSSSSSSSSS
\section{The quasi-static mobility boundary value problem}
\label{s:bvp}

Since we are in the viscous regime, there is no inertia,
and the particle velocities $\{\v_j(t),\omega_j(t)\}_{j=1}^{N_o}$ are
determined entirely by their current locations
$\{\x_j^c(t),\theta_j(t)\}_{j=1}^{N_o}$, via solving a
quasi-static {\em mobility} problem with zero applied forces and torques.
To phrase this as a BVP we need some basic definitions in Stokes flows.
We let $\u(\x)=(u_1(\x),u_2(\x))$ be the fluid velocity and $p(\x)$ be the pressure in the suspending fluid $\RR^2\backslash \overline{\Omega}_{\Lambda}$, and let $\bsigma$ be the stress tensor associated with the flow:
\begin{equation}
\sigma_{ij}(\u,p)=-\delta_{ij} p +\mu(\frac{\partial u_i}{\partial x_j}+\frac{\partial u_j}{\partial x_i})
=-\delta_{ij} p +\mu e(\u), \qquad i,j=1,2,
\end{equation}
where $\delta_{ij}$ is the Kronecker delta, and $e(\u):=\frac{1}{2}(D\u+D\u^{T})$ is the strain tensor. The
Einstein convention of summation is used here and below.

The hydrodynamic traction $\T(\u,p)$ (force vector per unit length that a boundary surface with outward unit
normal vector $\n$ applies to the fluid) is given by:
\begin{equation}
\T(\u,p)=\bsigma\cdot \n = \left[\begin{array}{cc}
\sigma_{11} & \sigma_{12}\\
\sigma_{21} & \sigma_{22}
\end{array}\right]\left[\begin{array}{c}
n_{1}\\
n_{2}
\end{array}\right]~.
\end{equation}
For notational convenience, we also define 
$\mathbf{x}^\perp = \left[\begin{array}{c}
-x_{2}\\x_{1} \end{array}\right]$ and $\nabla^{\perp} = \left[\begin{array}{c}
-\frac{\partial }{\partial x_{2}}\\ \frac{\partial }{\partial x_{1}} 
\end{array}\right] \, .$

We seek solutions $(\u,p)$ and $\{(\v_i,\omega_i)\}_{i=1}^{N_o}$ to the following BVP:
\begin{align}
-\mu\Delta\mathbf{u}+\nabla p & =0 \quad (\x\in \RR^2\backslash \overline{\Omega}_{\Lambda}) \label{eq:StokesFlowEq}\\
\nabla\cdot\mathbf{u} & =0 \quad (\x\in \RR^2\backslash \overline{\Omega}_{\Lambda})\label{eq:MassConservation}\\
\mathbf{u}\left(\mathbf{x}\right)+\mathbf{u}_0\left(\mathbf{x}\right) & =\mathbf{v}_{i}+\omega_{i}\left(\mathbf{x} - \mathbf{x}^{c}_i\right)^{\perp} \quad (\x\in \Gamma_i,\,i=1,\cdots,N_o) \label{eq:RigidBodyMotion}\\
\int_{\Gamma_{i}}\mathbf{T}(\u,p)\,ds_{\bx} & =\mathbf{0}\label{eq:ForceCondition}\\
\int_{\Gamma_{i}}  
( \mathbf{T}(\u,p),
(\mathbf{x} -\mathbf{x}^{c}_i)^{\perp}) \, ds_{\bx} 
& =0\label{eq:TorqueCondition} \\
\u(\x+m_1\e_1+m_2\e_2) &= \u(x) \quad (m_1,m_2 \in \ZZ) \label{eq:PerBCVelocity} \\
p(\x+m_1\e_1+m_2\e_2) &= p(x) \quad (m_1,m_2 \in \ZZ)~.
\label{eq:PerBCPressure}
\end{align}
Here $(\cdot,\cdot)$ is the Euclidean inner product for vectors in $\RR^2$. $\u_0(\x)=\left[\begin{array}{c} \gamma x_2 \\ 0 \end{array}\right]$ is the background shear flow matching the shear rate $\gamma$ of the lattice.
A valid pressure to associate with it to give a Stokes solution
is simply $p_0\equiv 0$.
$(\u,p)$ can be viewed as the perturbation of the velocity and pressure field, which, once obtained,
can be added to $(\u_0,p_0)$ to recover the physical velocity and pressure $\tilde{\u}=\u+\u_0$ and $\tilde{p}=p+p_0$. 
Both $(\u_0,p_0)$ and $(\u,p)$ satisfy the governing equations of Stokes flows \eqref{eq:StokesFlowEq} and \eqref{eq:MassConservation}.
$\v_i$, $\omega_i$, and $\x_i^c$ are the linear velocity, angular velocity and centroid of the $i$th object
$\Omega_i$.
Equation \eqref{eq:RigidBodyMotion} enforces rigid body motion of $\Omega_i$.
Equations \eqref{eq:ForceCondition} and \eqref{eq:TorqueCondition} state that the net force and torque
(induced by $(\u,p)$) on the boundary $\Gamma_i$ are zero. It can be verified directly that $(\u_0,p_0)$ does not contribute
to the net force nor torque, i.e. $\int_{\Gamma_{i}}\mathbf{T}(\u_0,p_0)\,ds_{\bx}  =\mathbf{0}$
and $\int_{\Gamma_{i}}(\mathbf{T}(\u_0,p_0),(\mathbf{x} -\mathbf{x}^{c}_i)^{\perp}) \, ds_{\bx}=0$, 
which implies that the physical force and torque (induced by $(\tilde{\u},\tilde{p})$) are also zero.
Equations \eqref{eq:PerBCVelocity} and
\eqref{eq:PerBCPressure} enforce lattice periodicity of the solution. 
\begin{rmk}
The periodicity boundary condition \eqref{eq:PerBCVelocity} and \eqref{eq:PerBCPressure} is equivalent to
\begin{align}
\u_R-\u_L &= \mathbf{0} \label{eq:discrepUrl} \\
\T(\u,p)_R-\T(\u,p)_L &= \mathbf{0} \label{eq:discrepFrl} \\
\u_U-\u_D &= \mathbf{0} \label{eq:discrepUtb} \\
\T(\u,p)_U-\T(\u,p)_D &= \mathbf{0}, \label{eq:discrepFtb} 
\end{align}
where the notation $\u_R$ is used to mean the restriction of $\u$ to the wall $R$. It is straightforward to
show that equations \eqref{eq:PerBCVelocity}--\eqref{eq:PerBCPressure} imply \eqref{eq:discrepUrl}--\eqref{eq:discrepFtb}. The converse follows by the unique continuation of Cauchy data $(\u,\T)$ as a solution to the 2nd-order Stokes PDE.
\end{rmk}
%-----------------------------------------------

Once the solution $(\u,p)$ is obtained at a given time $t$,
the effective viscosity can be retrieved (in the case where no object intersects
the bottom wall $D$) by
\begin{equation} 
  \mu_{\tbox{eff}}(t)\; \defeq \;
  \frac{1}{\gamma|\e_1|}\int_{D}
  \mbf{t} \cdot   \T(\u+\u_0,p) \,ds_{\x}
  ~,
  \label{mueff}
\end{equation}
where $\mbf{t}=(1,0)$ is the unit tangent vector on $D$.
Its interpretation is simply the total
horizontal (shear) force transmitted through the bottom wall by the fluid.
This formula bypasses the more complicated extraction of $\mu_\tbox{eff}(t)$
as a cell average of stress common in chemical engineering
\cite[Eq.~(8)]{fossbrady}.
In Sec.~\ref{s:mueff} we will present a variant of \eqref{mueff} that is
more efficient to evaluate, and which also is valid when particles
intersect $D$.

We can show that the BVP has a three-dimensional nullspace, which relies on the following lemma.
\begin{lem}
\label{lem:BoundaryTermZero}
If $(\u,p)$ satisfies \eqref{eq:StokesFlowEq}--\eqref{eq:PerBCPressure} with $\u_0=\mathbf{0}$ (homogeneous case), then
\begin{equation}
\int_{\Gamma_i} (\u,\T(\u,p))\,ds_{\x} =0~.
\end{equation}
\end{lem}
\begin{proof}
\begin{align*}
\int_{\Gamma_i} \left(\u, \T(\u,p)\right) ds_{\x} &= \int_{\Gamma_{i}} \left( \v_i+\omega_i\left(\x - 
\x^c_i \right)^{\perp},\T(\u,p)\right) ds_{\x} \\
&= \left(\v_i,\int_{\Gamma_i} \T(\u,p)\, ds_{\x}\right) + \omega_i\int_{\Gamma_i} \left((x-x_i^c)^{\perp},\T(\u,p) \right)\,ds_{\x} \\
& = \left(\v_i, \mathbf{0}\right) + \omega_i \cdot 0\, =0~.
\end{align*}
\end{proof}

Combining this with the divergence theorem, we get the following lemma on the uniqueness of the solution to the BVP.
\begin{lem}
If $(\u,p)$ satisfies \eqref{eq:StokesFlowEq}--\eqref{eq:PerBCPressure} with $\u_0=\mathbf{0}$, then
$\u(x)=const$, and $p(\x)=const$.
\label{l:bvpunique}
\end{lem}
\begin{proof}
Letting $\left\langle \cdot,\cdot\right\rangle :\bR^{2\times2}\times\bR^{2\times2}$
be the Frobenius inner product, we have
\begin{align*}
\int_{\cU\backslash \Omega_{\Lambda}}\left\langle e\left(\mathbf{u}\right),e\left(\mathbf{u}\right)\right\rangle dV 
& =\int_{\cU\backslash \Omega_{\Lambda}}\left\langle D\mathbf{u},e\left(\mathbf{u}\right)\right\rangle dV\\
 & \hspace{-.25in} =\int_{\partial\left(\cU\backslash \Omega_{\Lambda}\right)}\left(\mathbf{u},e\left(\mathbf{u}\right)\cdot\mathbf{n}\right)ds_{\bx}
-\frac{1}{2}\int_{\cU\backslash\Omega_{\Lambda}}\left(\mathbf{u},\Delta\mathbf{u}\right)dV\\
 & \hspace{-.25in} =\int_{\partial\left(\cU\backslash\Omega_{\Lambda}\right)}\left(\mathbf{u},e\left(\mathbf{u}\right)\cdot\mathbf{n}\right)ds_{\bx}
-\frac{1}{2\mu}\int_{\cU\backslash\Omega_{\Lambda}}\left(\mathbf{u},\nabla p\right)dV\\
 & \hspace{-.25in} =\frac{1}{2\mu}\int_{\partial\left(\cU\backslash\Omega_{\Lambda}\right)}\left(\mathbf{u},\left(-p\left[\begin{array}{cc}
1 & 0\\
0 & 1
\end{array}\right]+2\mu e\left(\mathbf{u}\right)\right).\mathbf{n}\right)ds_{\bx}\quad \\
 & \hspace{-.25in} =-\frac{1}{2\mu}\sum_{i=1}^{N_o}\int_{\Gamma_{i}}\left(\mathbf{u},\mathbf{T}\right)ds_{\bx}+\frac{1}{2\mu}\int_{\partial \cU}\left(\mathbf{u},\mathbf{T}\right)ds_{\bx}\\
 & \hspace{-.25in} =\frac{1}{2\mu}\int_{\partial \cU}\left(\mathbf{u},\mathbf{T}\right) ds_{\bx} \, = 0.
\end{align*}
The second line follows from Green's first identity, the fourth line follows from the divergence theorem,
and the last line combines the result of lemma \ref{lem:BoundaryTermZero} and the periodicity of $\u$ and $\T$.
% needs a little cleaning up...

As a result, we get:
\begin{equation}
e\left(\mathbf{u}\right)\equiv\left[\begin{array}{cc}
0 & 0\\
0 & 0
\end{array}\right]~, \qquad \mathbf{x} \in \cU\backslash \Omega_{\Lambda} \, .
\end{equation}
Thus, $\mathbf{u}$ is a rigid body motion, i.e. $\u(\x)=\v_0+\omega_0(\x-\x_0)^{\perp}$.

Letting $\x_L\in L$ and $\x_R=\x_L+\e_1 \in R$, we get:
\begin{align}
\u(\x_L) &= \v_0+\omega_0(\x_L-\x_0)^{\perp} \label{eq:uL} \\
\u(\x_L+\e_1) &= \v_0+\omega_0(\x_L+e_1-\x_0)^{\perp} \label{eq:uR} 
\end{align}

Subtracting \eqref{eq:uR} and \eqref{eq:uL} and recalling the periodicity in $\u$ leads to
$\omega_0 \e_1^{\perp} = 0 $, which means $\omega_0=0$, and thus $\u(\x)=\v_0$.
Substituting this into equation \eqref{eq:StokesFlowEq} yields $\nabla p=\mbf{0}$,
which further means $p=const$.
\end{proof}

\begin{rmk}
The existence of the solution of the BVP comes from that of the integral formulation and will be proved
in section~\ref{s:bie}.
\end{rmk}

% GGGGGGGGGGGGGGGGGGGGGGGGGGGGGGGGGGGGGGGGGGGGGGGGGGGGGGGGGGGGGGGGGGGGGGGGGGG
\section{Green's functions and Stokes layer potentials}
\label{s:green}

In this section we review some basic properties of the free space Stokes Green's function and layer potentials,
then propose a {\em generalized} periodic Green's function
that exists even in the case of non-zero net force in the unit cell.
Its construction, both analytically and numerically,
involves a direct sum of only the nearest free-space image sources,
plus an auxiliary Stokes solution that corrects for their
failure in periodicity.
When the net force is zero, as occurs in our application,
this Green's function will in fact be periodic.

\subsection{The free space Green's function and layer potentials}
For convenience we gather some
standard potential theory formulae for Stokes in 2D
\cite[Sec.~2.2, 2.3]{HW}, following the notation of \cite{ahb}.
The free space Green's function to the Stokes equation
(also called Stokeslet or single-layer kernel) is defined to be the tensor $G(\x,\y)$ with components:
\begin{equation}
G_{ij}(\x,\y) = \frac{1}{4\pi\mu}\left( \delta_{ij} \log \frac{1}{r}
+ \frac{r_ir_j}{r^2}\right),
\qquad i,j = 1,2, \quad \r := \x-\y, \quad r:=\|\r\|~.
\label{Gv}
\end{equation}
and the single-layer pressure kernel is the vector $G^p$ with components
\begin{equation}
G^p_j(\x,\y) = \frac{1}{2\pi}\frac{r_j}{r^2}
~,\qquad j=1,2~.
\label{Gp}
\end{equation}
We will also need the single-layer traction kernel $G^t$ with components
\begin{equation}
G^t_{ik}(\x,\y) = \sigma_{ij}(G_{\cdot,k}(\cdot,\y),G^p_k(\cdot,\y)) (\x) \nx_j
= -\frac{1}{\pi}\frac{r_ir_k}{r^2}\frac{\r\cdot\nx}{r^2}
~,\qquad i,k=1,2~,
\label{Gt}
\end{equation}
where the target $\x$ is assumed to be on a surface with normal $\nx$.

The Stokeslet allows us to express the velocity field $\u=(u_1,u_2)$ and pressure field $p$ induced by a point force
$\f=(\f_1,\f_2)$ at $\y_0=(y_{01}, y_{02})$ as:
\begin{equation}
u_i=G_{ij}(\x,\y_0) \f_j,\qquad p=G^p_{j}(\x,\y_0) \f_j.
\end{equation}

Let $\Gamma$ be a smooth closed curve in $\RR^{2}$ and let $\Omega^{\mp}$
denote the domains corresponding to the interior and exterior of 
$\Gamma$. Let $\n_{\x_0}$ be the unit outward normal to the curve $\Gamma$ at $\x_{0}\in\Gamma$ .
The single layer Stokes potential represents the velocity field $\u(\x)$ and pressure field $p(\x)$
induced by a surface force $\bsigma(\x)=(\sigma_1(\x),\sigma_2(\x))$ on a boundary $\Gamma$ as:
\begin{equation}
\u(\x)=\cs_{\Gamma}[\bsigma](\x)=\int_{\Gamma} G(\x,\y) \bsigma(\y)\,ds_{\y},\qquad p(\x)=\cs^p_{\Gamma}[\bsigma]=\int_{\Gamma} G^p(\x,\y) \bsigma(\y)\,ds_{\y}.
\end{equation}

The associated traction is given by:
\begin{equation}
\T(\x)=\cs^t_{\Gamma}=\int_{\Gamma} G^t(\x,\y) \bsigma(\y)\,ds_{\y}.
\end{equation}

It is well known in classical potential theory that the single layer potential has the following properties.

\begin{lem}{(the jump relation)} \label{lem:jumprelation}
Let $\u(\x)=\cs_{\Gamma}[\bsigma](\x)$ and $\T(\x)=\cs^t_{\Gamma}[\bsigma](\x)$ be the velocity and traction of a single layer potential with density $\bsigma$ defined on $\Gamma$. Then
$\cs_{\Gamma}\bsigma\left(\x\right)$ satisfies the Stokes equations \eqref{eq:StokesFlowEq} and \eqref{eq:MassConservation} in $\RR^{2}\backslash\Gamma$ and continuous
in $\RR^{2}$. The single layer traction satisfies
the jump relations:
\begin{equation}
\lim_{\substack{\x\to\x_0 \\ \x\in \Omega^{\pm}}} 
\T(\x)=
\mp\frac{1}{2}\bsigma\left(\x_0\right)+
\oint_{\Gamma} G^t(\x_0,\y)\bsigma\left(\y\right)ds_{\y} \label{SLnormaljump}
\end{equation}
where $\oint_{\Gamma}$ indicates the principal value integral
over the curve $\Gamma$.
\end{lem}

\begin{lem} \label{lem:netquantities}
Let $\u(\x)=\cs_{\Gamma}[\bsigma](\x)$ and $\T(\x)=\cs^t_{\Gamma}[\bsigma](\x)$ be the velocity and traction of a single layer potential with density $\bsigma$ defined on $\Gamma$.
Let $\T_{-}$ and $\T_{+}$ denote the limits of the traction from the interior and exterior respectively.
We have the following:
\begin{equation}
\int_{\Gamma}\T_{+} ds_{\x} =
-\int_{\Gamma}\bsigma\left(\x\right)ds_{\bx} \, ,\quad
\int_{\Gamma}\T_{-} \, ds_{\x} = \mathbf{0}
\label{eq:SLGaussLaw1StokesForce}
\end{equation}
and
\begin{align}
\int_{\Gamma}\left( \left(\x-
\x^c\right)^{\perp} ,\T \right)_{+} ds_{\bx}  & =
-\int_{\Gamma}\left( \left(\x-\x^c\right)^{\perp} ,
\bsigma \right)ds_{\bx}\label{eq:SLGaussLaw1StokesTorque} \, ,\\
\int_{\Gamma}\left( \left(\x-\x^c\right)^{\perp} ,
\T \right)_{-} ds_{\x}  & =0 \label{eq:SLGaussLaw2StokesTorque}\, .
\end{align}
and
\begin{equation}
\int_{\Gamma} \u\cdot\n\, ds_{\x}=0\,, \label{eq:SLNetFlux}
\end{equation}
which can be interpreted as the net force, torque and flux respectively.
\end{lem}

\subsection{A generalized periodic Stokes Green's function}
\label{subsec:gper}

We now must generalized the concept of a periodic Green's function,
since a single Stokeslet source (having non-zero net force)
cannot be periodized.
This fact is analogous to the non-existence of a domain Green's function
for the interior Laplace Neumann problem;
yet, in that setting there is a ``Neumann function'' \cite{kelloggbook}
which, when applied to zero-mean data,
recovers a solution obeying the Neumann boundary conditions.
Our generalized Green's function is analogous to the Neumann function:
it distributes the net force in a prescribed fashion on the boundary,
so that when the net force is non-zero, the traction is not periodic across unit
cell walls (Remark~\ref{r:notper}).
One motivation is rigorous analysis; another is
to numerically handle non-zero net force case in a predictable,
linear way, for stability in the iterative solver.

\begin{definition}  % dddddddddddddddddddddddddddddddddddd
  Let $\cU$ be a unit cell with four walls comprising its boundary $\partial\cU = L\cup R\cup U\cup D$.
  For each source point $\y\in\cU$,
  the generalized periodic Green's function, as a function of $\v\in\cU$,
  is defined as $G_{per}(\x,\y)\defeq [\w_{(1,0)}(\x), \; \w_{(0,1)}(\x)]$,
  where for each of the two choices of $\f$, the pair
  $(\w_{\f},q_\f)$ solves the Stokes BVP:
\begin{align}
-\mu\Delta \w+\nabla q &= \f \delta_\y \qquad \mbox{ in } \cU \label{eq:gperEqStokes}\\
\nabla\cdot \w &= 0  \qquad\;\; \mbox{ in } \cU \label{eq:gperEqMassConserv} \\
\w_R-\w_L &= \mbf{0}  \label{eq:gperdiscrep1} \\
\T(\w,q)_R-\T(\w,q)_L &= \mbf{f}/2|\e_2| \label{eq:gperdiscrep2} \\
\w_U-\w_D &= \mbf{0} \label{eq:gperdiscrep3} \\
\T(\w,q)_U-\T(\w,q)_D &= \mbf{f}/2|\e_1|\, . \label{eq:gperdiscrep4}
\end{align}
Here $\e_1$ and $\e_2$ are the unit cell lattice vectors.
The associated periodic pressure and traction kernels are defined as:
\begin{align}
G^p_{per}(\x,\y) &:= [q_{(1,0)}(\x),q_{(0,1)}(\x)] \\
G^t_{per}(\x,\y) &:= [\T(\w_{(1,0)},q_{(1,0)})(\x), \T(\w_{(0,1)},q_{(0,1)})(\x)]
~.
\end{align}
\end{definition}   % ddddddddddddddddddddddddddddddddddddddddd

We will prove later in this section that the generalized periodic Green's function $G_{per}$ is determined only
up to a constant $\c\in\RR^2$. Similarly $G^p_{per}$ is determined up tp a constant $c$ and $G^t_{per}$ is
determined up to a constant $c\n_{\x}$. $G_{per}$ is a smooth perturbation
of $G$, and the difference can be found by solving the following subproblem that we call the ``empty BVP.''
This is a periodic version of the construction of 
the domain Green's function in PDE theory \cite[Sec.~7.H]{Follandbook}.

For a given discrepancy vector $\g=[\g_1,\g_2,\g_3,\g_4]$, the empty BVP (EBVP)
is to find functions $(\v,q)$ solving:  
\begin{align}
-\mu\Delta \v+\nabla q &= \mbf{0} \qquad \mbox{ in } \cU \label{eq:emptyBoxEq1} \\
\nabla\cdot \v &= 0  \qquad\;\; \mbox{ in } \cU \label{eq:emptyBoxEq2} \\
\v_R-\v_L &= \mbf{g}_1  \label{eq:emptyBoxDiscrep1}\\
\T(\v,q)_R-\T(\v,q)_L &= \mbf{g}_2  \label{eq:emptyBoxDiscrep2}\\
\v_U-\v_D &= \mbf{g}_3 \label{eq:emptyBoxDiscrep3}\\
\T(\v,q)_U-\T(\v,q)_D &= \mbf{g}_4 \label{eq:emptyBoxDiscrep4}\, ,
\end{align}

The following extends the consistency and uniqueness result in \cite[Prop.~4.2]{ahb} to include existence, which was not proven in that work.
We are interested in smooth
solutions, which requires a technical condition that $\g$ is generated
from a smooth flow and pressure function pair $(\mbf{w},r)$.
Note that $(\mbf{w},r)$ need not satisfy the Stokes equations, and $\mbf{w}$
need not even be divergence-free.
The proof of the following is in \ref{a:EBVP}.

% EEEEEEEEEEEEEEEEEEEEEEEEEEEEEEEEEEEEEEEEEEEEEEEEEEEEEEEEEEEEEEEEEEEEEEEEEE
\begin{theorem} \label{t:EBVP}
  Let the discrepancy data $\g=[\g_1,\g_2,\g_3,\g_4]$
  be {\em smooth} in the sense that it is generated as follows from a
  $C^\infty(\overline{\cU})$ pair $(\mbf{w},r)$:
\begin{align}
  \g_1 &= \mbf{w}_R-\mbf{w}_L \label{g1} \\
  \g_2 &=    \T(\mbf{w},r)_R-\T(\mbf{w},r)_L \\
  \g_3 &=      \mbf{w}_U-\mbf{w}_D \\
  \g_4 &= \T(\mbf{w},r)_U-\T(\mbf{w},r)_D \label{g4}
  ~.
\end{align}
Let $\g$ also be {\em consistent} in the sense that three
scalar conditions hold:
\begin{align}
\int_L \mbf{g}_2 ds + \int_D \mbf{g}_4 ds &= \mbf{0} \qquad 
\mbox{{\rm (zero net force)}, and}
\label{gforce}
\\
\int_L n\cdot \mbf{g}_1 ds + \int_D n\cdot \mbf{g}_3 ds &= 0 \qquad 
\mbox{{\rm (volume conservation)}.}
\label{gvol}
\end{align}
Then there exists a $C^\infty(\cU)$ solution $(\v,q)$ to the EBVP
\eqref{eq:emptyBoxEq1}--\eqref{eq:emptyBoxDiscrep4},
and its nullity is 3.
Specifically, the solution is
unique up to translational flow and additive pressure constants,
with solution space $(\v +\mbf{c},q + c)$, for $(\mbf{c},c) \in \RR^3$.
\end{theorem}   % EEEEEEEEEEEEEEEEEEEEEEEEEEEEEEEEEEEEEEEEEEEEEEEEEEEEEEEEEEEE

Utilizing the solution to this subproblem as a building block, we now proceed to construct the periodic Green's
function, or equivalently, the solution of the BVP \eqref{eq:gperEqStokes}--\eqref{eq:gperdiscrep4} for a given point force $\f=(f_1,f_2)$ at $\y=(y_1,y_2)$.
We begin by defining a nearby image source sum, and a metric of the failure
to be periodic.

\begin{definition} \label{def:Gnear}
Let $\y\in\cU$ be a point in the unit cell, and $\f\in\RR^2$ be an arbitary vection.
Let $\tilde{\cU}$ be the expanded unit cell with the same center and walls parallel to that of the original cell
and twice the length (as is illustrated in Fig \ref{f:periebvp}) $w^{near}$ and $q^{near}$ are defined to be the sum of neighboring copies of the free space Stokeslets that are inside $\tilde{\cU}$, i.e.
\begin{align}
\w^{near}=G_{near}(\x,\y)\f &\defeq \sum_{m,n\in\ZZ: \y+m\e_1+n\e_2\in \tilde{\cU}} G(\x,\y+m\e_1+n\e_2)\f \\
q^{near}=G^{p}_{near}(\x,\y)\cdot\f &\defeq \sum_{m,n\in\ZZ: \y+m\e_1+n\e_2\in \tilde{\cU}} G^p(\x,\y+m\e_1+n\e_2)\cdot\f.
\end{align}
\end{definition}

\begin{definition}
Let $\tilde{\cU}$, $\y$ and $\f$ be the same as in Definition \ref{def:Gnear}. $w^{corr}$ and $q^{corr}$ are
defined to be the solution to the empty box BVP \eqref{eq:emptyBoxEq1}--\eqref{eq:emptyBoxDiscrep4}, with the ``discrepancy vector" $\g=[\g_1;\g_2;\g_3;\g_4]$ as follows:
\begin{align}
\g_1 &= -(\w^{near}_R-\w^{near}_L) \label{eq:discrepCorr1} \\
\g_2 &= -(\T^{near}_R-\T^{near}_L) +\f/2|\e_2| \label{eq:discrepCorr2} \\
\g_3 &= -(\w^{near}_U-\w^{near}_D) \label{eq:discrepCorr3} \\
\g_4 &= -(\T^{near}_U-\T^{near}_D) +\f/2|\e_1| \label{eq:discrepCorr4} \,.
\end{align}
\end{definition}

We now claim that the specific choice of $\g$ just defined leads
to existence of an EBVP solution $(\v,q)$. The following is
proved in \ref{a:cons}.

\begin{theorem} \label{t:consistency}   % ttttttttttttttttttttttttttttttttttt
  Let the discrepancy vector $\g=[\g_1;\g_2;\g_3;\g_4]$ be given by
  \eqref{eq:discrepCorr1}--\eqref{eq:discrepCorr4}.
  Then a solution $(\v,q)$ to \eqref{eq:emptyBoxEq1}--\eqref{eq:emptyBoxDiscrep4}
  exists and is unique up to additive constants in $\v$ and $q$. 
\end{theorem}    % ttttttttttttttttttttttttttttttttttttttttttttttttttttttttt

%fffffffffffffffffffff
\bfi  
\centering\ig{width=2.5in}{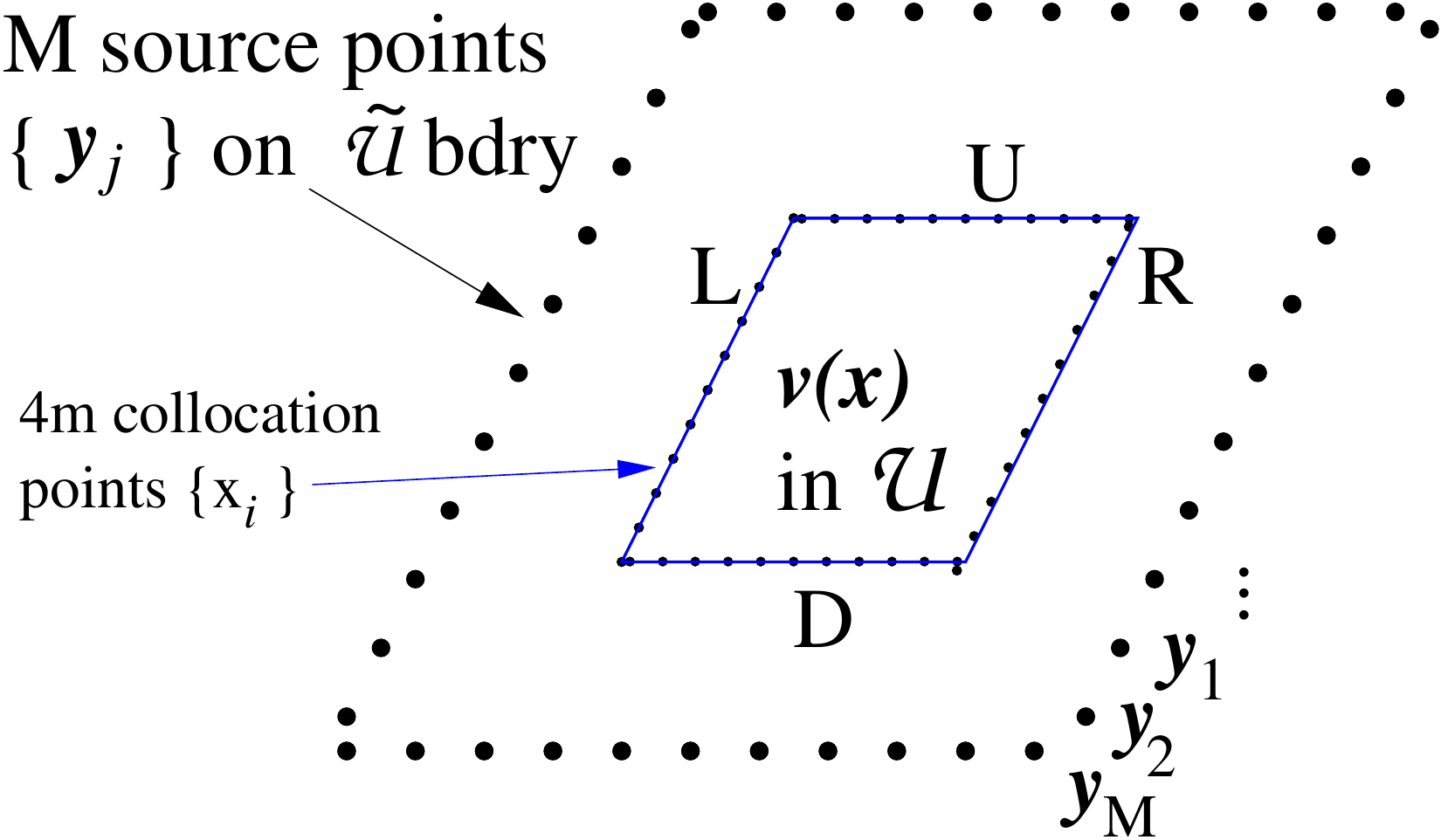}
\ca{Illustration of the solution method for the ``empty box'' BVP
which is used to correct the deviation from periodicity in the near image sum.}{f:periebvp}
\efi

Note that the inclusion of all the image sources lying in an dilation of the
unit cell in $\w^{near}$ and $q^{near}$ from Definition~\ref{def:Gnear}
actually results in cancellations of all nearby source contributions
to the discrepancy $\g$; see \cite{ahb}.
This results in $\g$ being smooth, so that the solution
of the ``empty BVP'' (section~\ref{s:empty})
will be rapidly convergent, needing only $\bigO(1)$ effort.
(We have not exploited this additional smoothness in the
proof in \ref{a:cons}.)

It is trivial to observe that both the ``near part" $(w^{near}, q^{near})$ and the ``correction part" $(w^{corr}, q^{corr})$ satisfy equations \eqref{eq:gperEqStokes}--\eqref{eq:gperEqMassConserv}, 
and that the discrepancy \eqref{eq:discrepCorr1}--\eqref{eq:discrepCorr2} of $(w^{corr}, q^{corr})$ exactly cancels that of $(w^{near}, q^{near})$. Consequently the sums $w=w^{near}+w^{corr}$ and $q=q^{near}+q^{corr}$ form a solution to the BVP \eqref{eq:gperEqStokes}--\eqref{eq:gperdiscrep4}. By taking $\f=(1,0)$ and $\f=(0,1)$, and
letting $G_{corr}=[w_{\y,(1,0)}^{corr}(\x),w_{\y,(0,1)}^{corr}(\x)]$, we obtain the useful decomposition:
\begin{equation} \label{eq:GperDecomp}
G_{per}(\x,\y)=G_{near}(\x,\y) + G_{corr}(\x,\y).
\end{equation}
The same decomposition holds for the pressure:
\begin{equation} \label{eq:GpperDecomp}
G_{per}^p(\x,\y)=G_{near}^p(\x,\y) + G_{corr}^p(\x,\y).
\end{equation}
This split is a useful tool for both the analysis and the numerical method, and will be revisited in later sections.

As in the free space case, we can define the periodic single layer potential induced by $\bsigma(\x)=(\sigma_1(\x),\sigma_2(\x))$ on a boundary $\Gamma$ as:
\begin{equation}
\u(\x)=\cs^{per}_{\Gamma}[\bsigma](\x)=\int_{\Gamma} G_{per}(\x,\y) \bsigma(\y)\,ds_{\y},\quad p(\x)=\cs^{per,p}_{\Gamma}[\bsigma]=\int_{\Gamma} G_{per}^p(\x,\y) \bsigma(\y)\,ds_{\y}.
\end{equation}

The associated traction is given by:
\begin{equation}
\T(\x)=\cs^{per,t}_{\Gamma}=\int_{\Gamma} G_{per}^t(\x,\y) \bsigma(\y)\,ds_{\y}.
\end{equation}

Equations \eqref{eq:GperDecomp} and \eqref{eq:GpperDecomp} motivate a similar decomposition for the single layer potential:
\begin{align}
\cs^{per}_{\Gamma}[\bsigma](\x) &= \cs^{near}_{\Gamma}[\bsigma](\x) +  \cs^{corr}_{\Gamma}[\bsigma](\x) \label{eq:uperdecomp} \\
\cs^{per,p}_{\Gamma}[\bsigma](\x) &= \cs^{near,p}_{\Gamma}[\bsigma](\x) +  \cs^{corr,p}_{\Gamma}[\bsigma](\x)\, , \label{eq:pperdecomp}
\end{align}
where $\cs^{near}_{\Gamma}[\bsigma](\x)$ and $\cs^{near,p}_{\Gamma}[\bsigma](\x)$ are the sums over the near
copies of the unit cell of $\cs_{\Gamma}[\bsigma](\x)$ and $\cs^{p}_{\Gamma}[\bsigma](\x)$ respectively.
And $(\cs^{corr}_{\Gamma}[\bsigma](\x), \cs^{corr,p}_{\Gamma}[\bsigma](\x))$ is a solution to the empty box BVP \eqref{eq:emptyBoxEq1}--\eqref{eq:emptyBoxDiscrep4}, with the discrepancy vector $\g=[\g_1;\g_2;\g_3;\g_4]$ given by:
\begin{align}
\g_1 &= -(\cs^{near}_{\Gamma,R}-\cs^{near}_{\Gamma,L}) \label{eq:SLPdiscrepCorr1} \\
\g_2 &= -(\cs^{near,t}_{\Gamma,R}-\cs^{near,t}_{\Gamma,L}) +\F/2|\e_2| \label{eq:SLPdiscrepCorr2} \\
\g_3 &= -(\cs^{near}_{\Gamma,U}-\cs^{near}_{\Gamma,D}) \label{eq:SLPdiscrepCorr3} \\
\g_4 &= -(\cs^{near,t}_{\Gamma,U}-\cs^{near,t}_{\Gamma,D}) +\F/2|\e_1| \label{eq:SLPdiscrepCorr4}\,,
\end{align}
where $\F=\int_{\Gamma} \bsigma\,ds$ is the net force.

\begin{rmk}\label{r:notper}
The ``periodic" single layer potential is doubly periodic in the
directions $\e_1$ and $\e_2$ if and only if $\F=\int_{\Gamma} \bsigma\,ds =\mbf{0}$. In the general case where
this condition is violated, the velocity field is doubly periodic while the traction has a constant jump
across the walls.
\end{rmk}

\begin{rmk}
The difference between the periodic single layer potential and its free space
counterpart is smooth in $\cU$. Thus the jump relation of the single layer potential as stated
in lemma \ref{lem:jumprelation} holds without any changes. In fact, the properties stated in 
lemma \ref{lem:netquantities} also generalize directly to the periodic case. To prove this statement, we require the following lemma.
\end{rmk}

\begin{lem}
Let $(\u^{corr},p^{corr})$ be the correction terms defined by \eqref{eq:uperdecomp} and \eqref{eq:pperdecomp}.
They do not contribute to the net force, net torque, or the net flux. More specifically, we have
\begin{align}
\int_{\Gamma} \T(\u^{corr},p^{corr}) \,ds &=0 \label{eq:NetForceCorr} \\
\int_{\Gamma} \left((\x-\x^c)^{\perp},\T(\u^{corr},p^{corr})\right) \,ds &=0 \label{eq:NetTorqueCorr} \\
\int_{\Gamma} \u^{corr}\cdot\n \,ds &=0 , \label{eq:NetFluxCorr}
\end{align}
where $\Gamma=\partial \Omega$, $\Omega\in\cU$ is any domain in the unit cell.
\end{lem}

\begin{proof}
Let $D=\cU\backslash\overline{\Omega}$, and recall the divergence theorem for the Stokes equation:
\begin{equation} \label{eq:DivThmStokes}
\int_{D} (-\mu\Delta \u^{corr}+\nabla p^{corr})\,d\x = -\int_{\partial D} \T(\u^{corr},p^{corr})\,ds.
\end{equation}
Since $(\u^{corr},p^{corr})$ satisfies the Stokes equation in $D$, the left hand side vanishes. And the
right hand side can be expanded into:
\begin{align}
\int_{\partial D} \T(\u^{corr},p^{corr})\,ds &= \int_{\partial \cU} \T(\u^{corr},p^{corr})\,ds -\int_{\Gamma} \T(\u^{corr},p^{corr})\,ds \\
&= -\int_{\Gamma} \T(\u^{corr},p^{corr})\,ds, 
\end{align}
which implies \eqref{eq:NetForceCorr}.

The proof of \eqref{eq:NetTorqueCorr} relies on the Green's first identity for the Stokes equation:
\begin{equation} \label{G1I}
\int_{\cal K} (\mu \Delta u_i - \partial_i p) v_i
= - \frac{\mu}{2} \int_{\cal K} (\partial_i u_j + \partial_j u_i)(\partial_i v_j + \partial_j v_i) + \int_{\partial\cal K} T_i(\u,p) v_i ~,
\end{equation}
where ${\cal K}$ is an arbitary domain. By taking ${\cal K}=D$, $(\u,p)=(\u^{corr},p^{corr})$ and $\v=(\x-\x^c)^{\perp}$, we have:
\begin{align}
\partial_i v_j + \partial_j v_i &= 0,\, i,j=1,2 \\
\int_D (\mu \Delta u_i - \partial_i p) v_i &= 0 \\
\int_{\partial \cU} T_i(\u,p) v_i &= 0,
\end{align}
leading to $\int_{\Gamma} \left((\x-\x^c)^{\perp},\T(\u^{corr},p^{corr})\right) \,ds =0$.

\eqref{eq:NetFluxCorr} is a direct consequence of the divergence theorem and the periodicity of $\u^{corr}$:
\begin{equation}
\int_{\Gamma} \u^{corr}\cdot\n\,ds = \int_{\partial\cU} \u^{corr}\cdot\n\,ds = \int_{\cU} \nabla\cdot\u^{corr}\,d\x =0.
\end{equation}
\end{proof}

% IIIIIIIIIIIIIIIIIIIIIIIIIIIIIIIIIIIIIIIIIIIIIIIIIIIIIIIIIIIIIIIIIIIIIIIIIII
\section{Integral equation representation of the mobility BVP}
\label{s:bie}

The previous section enables us to seek a solution to the quasi-static BVP in the form of a periodic single layer
potential. 
We assume that on each boundary $\Gamma_i$, it is induced by an unknown density function $\bsigma_i(\x)$.
Here each density $\bsigma_i(\x)$ is a vector function $\bsigma_i(\x)=(\sigma_{i1}(\x),\sigma_{i2}(\x))$. 
We use the condensed notation
\begin{equation}
\bsigma(\x)=(\sigma_{11}(\x),\sigma_{12}(\x),\cdots,\sigma_{N_o1}(\x),\sigma_{N_o2}(\x)),
\end{equation}
and 
\begin{align}
\u(\x) &= \cs_{\Gamma}^{per}[\bsigma](\x)\defeq \sum_{i=1}^{N_o} \cs_{\Gamma_i}^{per}[\bsigma_i](\x) \label{eq:BIEurep} \\
p(\x) &= \cs_{\Gamma}^{per,p}[\bsigma](\x)\defeq \sum_{i=1}^{N_o} \cs_{\Gamma_i}^{per,p}[\bsigma_i](\x) \label{eq:BIEprep} \\
\T(\x) &= \cs_{\Gamma}^{per,t}[\bsigma](\x)\defeq \sum_{i=1}^{N_o} \cs_{\Gamma_i}^{per,t}[\bsigma_i](\x) \label{eq:BIEtrep} \,,
\end{align}

Since there is no inertial stress in a rigid body, the stress tensor must be identically zero inside $\Omega_i$.
This motivates us to reinterpret the boundary condition as
\begin{equation}
\T_{-}(\x)=(\sigma(\x)\cdot \n(\x))_{-} \equiv \mbf{0},\qquad (\x\in\Gamma_i)
\end{equation}

Recalling the jump relation \eqref{SLnormaljump} (the periodic case),
we obtain the following Fredholm integral equation of the second kind:
\begin{equation} \label{eq:BIEnoConstraint}
\left(\frac{1}{2}I+\cK\right)\bsigma(\x)=-\T(\u_0(\x),p_0(\x)),\qquad \x\in\Gamma,
\end{equation}
where 
\begin{equation}
\mathbf{I} =\left(\begin{array}{cccc}
\mathbf{I}_{1}\\
 & \mathbf{I}_{2}\\
 &  & \ddots\\
 &  &  & \mathbf{I}_{N_o}
\end{array}\right) 
\end{equation}
and
\begin{equation}
\mathcal{K} =\left(\begin{array}{cccc}
\mathcal{K}_{1,1} & \mathcal{K}_{1,2} & \ldots & \mathcal{K}_{1,N_o}\\
\mathcal{K}_{2,1} & \mathcal{K}_{2,2} & \ldots & K_{2,N_o}\\
\vdots & \vdots & \ddots & \vdots\\
\mathcal{K}_{N_o,1} & \mathcal{K}_{N_o,2} &  & \mathcal{K}_{N_o,N_o}
\end{array}\right).
\end{equation}

Here each $\cK_{ij}$ is the periodic single layer traction operator from the j-th boundary component to
the i-th boundary component, which is given by:
\begin{equation}
\cK_{ij}[\bsigma_j](\x)=\cs_{\Gamma_j}^{per,t}[\bsigma_j](\x)=\int_{\Gamma_j} G_{per}^t(\x,\y)\bsigma_j(\y)\,ds_{\y},
\qquad \x\in\Gamma_i,\, i\neq j,
\end{equation}
and 
\begin{equation}
\cK_{ii}[\bsigma_i](\x)=\cs_{\Gamma_i}^{per,t}[\bsigma_j](\x)=\oint_{\Gamma_i} G_{per}^t(\x,\y)\bsigma_i(\y)\,ds_{\y},
\qquad \x\in\Gamma_i\, .
\end{equation}

Equations \eqref{eq:SLGaussLaw1StokesForce}, \eqref{eq:SLGaussLaw1StokesTorque} and \eqref{eq:SLGaussLaw2StokesTorque}
(the periodic version) enable us to reformulate the force and torque constraints as:
\begin{align}
\int_{\Gamma_i} \bsigma_i(\x)\,ds_{\x} &= \mbf{0} \label{eq:BIEconstraint1}  \\
\int_{\Gamma_i} \left((\x-\x_i^c)^{\perp},\bsigma_i(\x)\right)\,ds_{\x} &= o. \label{eq:BIEconstraint2}
\end{align}

For notational convenience, we further introduce the operator $L_i$ defined on $\Gamma_i$ as:
\begin{equation}
L_i\bsigma_i(\x)=\int_{\Gamma_i}\bsigma_i(\y)\,ds_{\y}
+(\x-\x_i^c)^{\perp}\int_{\Gamma_i}\left((\y-\x_i^c)^{\perp},\bsigma_i(\y)\right) ds_{\y},
\end{equation}
and define $L$ as:
\begin{align}
\mathbf{L} & =\left(\begin{array}{cccc}
\mathbf{L}_{1}\\
 & \mathbf{L}_{2}\\
 &  & \ddots\\
 &  &  & \mathbf{L}_{N_o}
\end{array}\right)\,.
\end{align}

It is straightforward to verify that if $\bsigma(\x)$ satisfies equations \eqref{eq:BIEnoConstraint},
\eqref{eq:BIEconstraint1} and \eqref{eq:BIEconstraint2}, then $(\u,p)=\left(\cs_{\Gamma}^{per}[\bsigma], \cs_{\Gamma}^{per,p}[\bsigma]\right)$ solves the BVP \eqref{eq:StokesFlowEq}--\eqref{eq:PerBCPressure}.
In particular, \eqref{eq:BIEconstraint1} implies that the
total force $\sum_{i=1}^{N_o} \int_{\Gamma_i} \bsigma_i ds_{\x} = \mbf{0}$,
so that $\u$ generated by the generalized Green's function is actually
periodic.
Here the operator $(\frac{1}{2}I+\cK)$ has a $3N_o$-dimensional null space.
The purpose of the $3N_o$
constraints is to select the particular solution with zero net forces and torques on all bodies.
However this constrained BIE, when discretized, leads to a rectangular linear system that is less favorable
than a square one. Here, following \cite{karrilakim,rachhgreengard} we propose instead to solve the following integral equation:
\begin{equation} \label{eq:BIEPerturbed}
\left(\frac{1}{2}I+\cK+L\right)\bsigma(\x)=-\T(\u_0(\x),p_0(\x))\,,
\end{equation}
or componentwise:
\begin{align}
\frac{1}{2}{\bsigma}_{i}\left(\x\right) &+\sum_{j=1}^{N_o}
\cK_{i,j}\bsigma_{j}(\x)
+\int_{\Gamma_i}\bsigma_i (\y)ds_{\y} \nonumber \\
&+ \left(\x - \x_i^c \right)^{\perp}
\int_{\Gamma_i} \left( \left(\y - \x^c_i \right)^{\perp}, \bsigma_{i} (\y) \right) ds_{\y}
 = -T(\u_0(\x),p_0(\x))
\quad \x\in\Gamma_i\label{eq:BIEPerturbed2}
\end{align}

It is trivial to verify that equations \eqref{eq:BIEnoConstraint},
\eqref{eq:BIEconstraint1} and \eqref{eq:BIEconstraint2} imply \eqref{eq:BIEPerturbed}. 
\cite{rachhgreengard} has proved in the free space case that they are in fact equivalent. The same statement holds for the
periodic case. Here we state it as a lemma and leave the proof, which is similar to the free space case,
to the reader.

\begin{lem}
If $\bsigma$ solves \eqref{eq:BIEPerturbed}, then it solves \eqref{eq:BIEnoConstraint},
\eqref{eq:BIEconstraint1} and \eqref{eq:BIEconstraint2}.
\end{lem}

The next theorem establishes the existence and uniqueness of the boundary integral equation \eqref{eq:BIEPerturbed}, which further implies the existence of a solution to the quasi-static BVP.

\begin{theorem}
The boundary integral equation \eqref{eq:BIEPerturbed} has a unique solution.
\end{theorem}

\begin{proof}
  The operator $\frac{1}{2}I+\cK+L$ is injective. By the Fredholm alternative,
  it is bijective. 
\end{proof}

\begin{theorem}
There exists a solution to the quasi static BVP 
\eqref{eq:StokesFlowEq}--\eqref{eq:PerBCPressure}.
\label{t:bvp}
\end{theorem}

\begin{proof}
Let $\bsigma$ be a solution to the BIE \eqref{eq:BIEPerturbed}. $\u(\x) = \cs_{\Gamma}^{per}[\bsigma](\x)$,
$p(\x) = \cs_{\Gamma}^{per,p}[\bsigma](\x)$ form a solution to the BVP.
\end{proof}

When combined with Lemma~\ref{l:bvpunique},
this completely characterizes existence and uniqueness for the
quasi-static BVP.

% NNNNNNNNNNNNNNNNNNNNNNNNNNNNNNNNNNNNNNNNNNNNNNNNNNNNNNNNNNNNNNNNNNNNNNNNNNN
\section{Numerical solution of the BIE}
\label{sec:numerical}

In this section, we discuss the numerical solution of the BIE \eqref{eq:BIEPerturbed2}.
We assume that on each boundary component $\Gamma_i$, we are given a set of quadrature nodes $\{\x_i^{(k)}\}_{k=1}^{N_i}$ and weights $\{w_i^{(k)}\}_{k=1}^{N_i}$ such that
\begin{equation*}
\int_{\Gamma_i} f(\x)\,ds_{\x} \approx \sum_{k=1}^{N_i} f(\x_i^{(k)})w_i^{(k)}
\end{equation*}
holds to high accuracy for smooth functions $f$. For example, if the boundary $\Gamma_i$ is parametrized by a
$2\pi$-periodic function $\x_i(t),\,0\leq t\leq2\pi$, then the periodic trapezoidal rule in $\t$ gives
$\x_i^{(k)}=\x_i(2\pi k/N_i)$, and $w_i^{(k)}=(2\pi/N_i)||\x_i^{'}(2\pi k/N_i)||$.

Since the operators $\cK$ and $L$ appearing in the equation are both smooth, we apply the standard Nystr\"om
discretization to the BIE, i.e. enforce \eqref{eq:BIEPerturbed2} at the same quadrature nodes
$\{\x_i^{(k)}: k=1,\cdots,N_i,\, i=1,\cdots,N_o \}$ to get:
\begin{align} 
\frac{1}{2}\bsigma_i^k &+\sum_{j=1}^{N_o} \sum_{l=1}^{N_j} G_{per}^t (\x_k^{(i)},\x_l^{(j)})\bsigma_j^l\cdot w_l^{(j)}
+\sum_{m=1}^{N_i} \bsigma_i^m \cdot w_m^{(i)} \nonumber \\
&+(\x_k^{(i)}-\x_i^c)^{\perp} \cdot
\sum_{m=1}^{N_i} \left((\x_m^{(i)}-\x_i^c)^{\perp},\bsigma_i^m\right)\cdot w_m^{(i)} \nonumber \\
&= \T(\u_0(\x_k^{(i)},p_0(\x_k^{(i)}))),\qquad 
(k=1,\cdots,N_i,\, i=1,\cdots,N_o)\,, \label{eq:BIEdiscretized}
\end{align}
where $\bsigma_i^k = (\sigma_{i1}^k, \sigma_{i2}^k)\defeq \bsigma_i(\x_i^{(k)})=(\sigma_{i1}(\x_i^{(k)}),\sigma_{i2}(\x_i^{(k)}))$.

To complete the numerical method, several problems remain to be addressed. We discuss them in details in the
following subsections.

\subsection{Numerical methods for the empty box BVP}
\label{s:empty}

As is mentioned in previous sections, numerical solution of the empty box BVP \eqref{eq:emptyBoxEq1}--\eqref{eq:emptyBoxDiscrep4} is a necessary correction term in the
evaluation of the periodic Green's function. There exist an abundance of methods for this problem. We find it simple, efficient, and accurate to
solve it by the method of fundamental solutions (MFS).

Letting $\phi_j(\x)\defeq G(\x,\y_j)$ and $ \phi_j^p(\x)\defeq G^p(\x,\y_j)$, where $G(\x,\y)$ and $G^p(\x,\y)$
are the free space Green's functions for velocity and pressure, and $\{\y_j\}_{j=1}^M$ are a collection of points
lying on the walls of the expanded unit cell $\tilde{\cU}$, We approximate the solution in $\cU$ by a linear
combination of such functions. That is, for $\x\in\cU$,
\begin{align}
  \v(\x) & \approx \sum_{j=1}^M \phi_j(\x)\bxi_j
\label{vrep}
  \\
  q(\x) & \approx \sum_{j=1}^M \phi_j^p(\x)\cdot \bxi_j \,,
  \label{qrep}
\end{align}
where $\{\bxi_j\}_{j=1}^M$ are unknown vectors in $\RR^2$. Since each basis function is a solution to the Stokes
equation \eqref{eq:emptyBoxEq1}--\eqref{eq:emptyBoxEq2}, it remains only to enforce the boundary conditions
\eqref{eq:emptyBoxDiscrep1}--\eqref{eq:emptyBoxDiscrep4}.

To enforce the boundary conditions, we let $\{\x_{iL}\}_{j=1}^m \subset L$ and $\{\x_{iD}\}_{j=1}^m \subset D$ be two sets of $m$ collocation points
(Gauss Legendre nodes, for example) on the left and bottom walls respectively.
Enforcing condition \eqref{eq:emptyBoxDiscrep1} on the left wall gives
\begin{equation}
\sum_{j=1}^M \left[\phi_j(\x_{iL}+\e_1)- \phi_j(\x_{iL})\right] \bxi_j = \g_1(\x_{iL}), \,\, i=1,\cdots,m.
\end{equation}
We have $2M$ unknowns in the coefficient vector $\xi\defeq \{\bxi_j\}_{j=1}^M$, and $2m$ equations with right
hand side $d_1 \defeq \{\g_1(\x_{iL})\}_{i=1}^m$. We order the vectors with all the 1-components followed
by all the 2-components.  This leads to a linear system $Q_1\xi=d_1$, where $Q_1=[Q_1^{11}, Q_1^{12}; Q_1^{21}, Q_1^{22}]$, with each block having entries $\left(Q_1^{kl}\right)_{ij}=\phi_j^{kl}(\x_{iL}+\e_1)- \phi_j^{kl}(\x_{iL})$. 

Repeating this bookkeeping routine for the boundary conditions \eqref{eq:emptyBoxDiscrep2}--\eqref{eq:emptyBoxDiscrep4}, we get the linear system 
\begin{equation}
Q\xi = d, 
\end{equation}
where $Q=[Q_1;Q_2;Q_3;Q_4]$ and $d=[d_1;d_2;d_3;d_4]$, and each block corresponds to one of the boundary
conditions. This linear system (which is generally rectangular and ill conditioned) is then solved in
the least square sense. A convergence test is given in Fig.~\ref{f:ebvp_conv}.

\begin{rmk}
The method we presented here is similar to that in \cite{ahb}. An improvement is that we now choose the walls
of the expanded unit cell as the ``proxy surface", which is generic and works more robustly for skewed 
unit cells with a wide range of aspect ratios.
\end{rmk}

The system matrix $Q$ inherits the consistency conditions \eqref{gforce} and \eqref{gvol}, which in terms of linear algebra, can be stated
as
\begin{equation}
W^{T}Q \approx \bm{0}_{3\times 2M},
\end{equation}
where
\begin{equation}
W^T = \left[\begin{array}{llllllll}
0 & 0 & \w_L^T & 0 & 0 & 0 & \w_D^T & 0\\
0 & 0 & 0 & \w_L^T & 0 & 0 & 0 & \w_D^T \\
\w_L^T\e_2^1 & \w_L^T\e_2^2 & 0 & 0 & \w_D^T\e_1^1 & \w_D^T\e_1^2 & 0 & 0
\end{array}\right]
\; \in \; \RR^{3\times 8m}
~.
\label{W}
\end{equation}

{\bf Example 1.} As a verification of the method, we solve the EBVP with the discrepancy vector $\g$ generated
by a known solution, in unit cells with different aspect ratios. More specifically, we fix $\mu=0.7$, 
$\y_0=1.5(\cos{0.1}, \sin{0.1})$, and $\f_0=(0.3,-0.6)$, and let $\v(\x)=G(\x,\y_0)\f_0$,
and $q(\x)=G^p(\x,\y_0)\f_0$. In the first example, we choose the unit cell to be the unit
square centered at the origin with $\e_1=(1,0)$ and $\e_2=(0,1)$. In the second example, we change $\e_2$ to be
$(\cos{\frac{\pi}{4}}, \sin{\frac{\pi}{4}})$, while keeping the center and $\e_1$ unchanged.

In both cases, $\g$ is genearted by evaluating $\v(\x)$ and $\T(\v(\x),q(\x))$ on the walls, and the resulting
EBVP is solved by the MFS. The solution is computed on a $100\times 100$ grid in the unit cell and the error
is obtained by comparing the computed solution to the exact one. The error at the first grid point is subtracted
to account for the fact that the solution is unique only up to a constant.

The results are given in Fig \ref{f:ebvp_conv} which shows the convergence in the
number $m$ of the collocation points on each wall, and in the number $M$ of proxy points. In both cases, we
observe an exponential convergence in both $m$ and $M$. The numbers required to achieve a given accuracy are
relatively insensitive to the aspect ratio. 

\bfi  % fffffffffffffffffffff
\ig{width=1.95in}{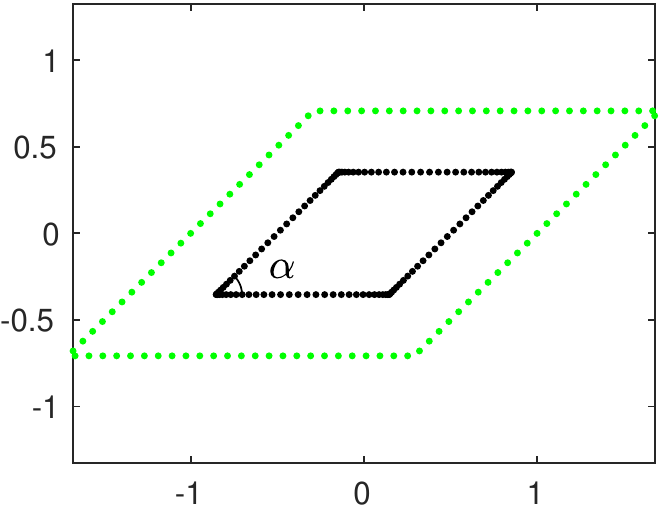}
\ig{width=1.95in}{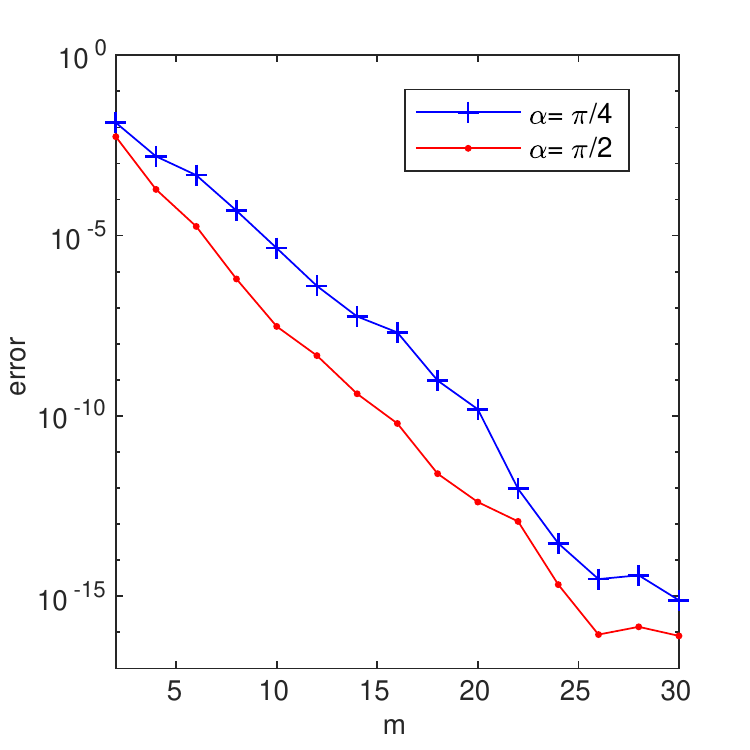}
\ig{width=1.95in}{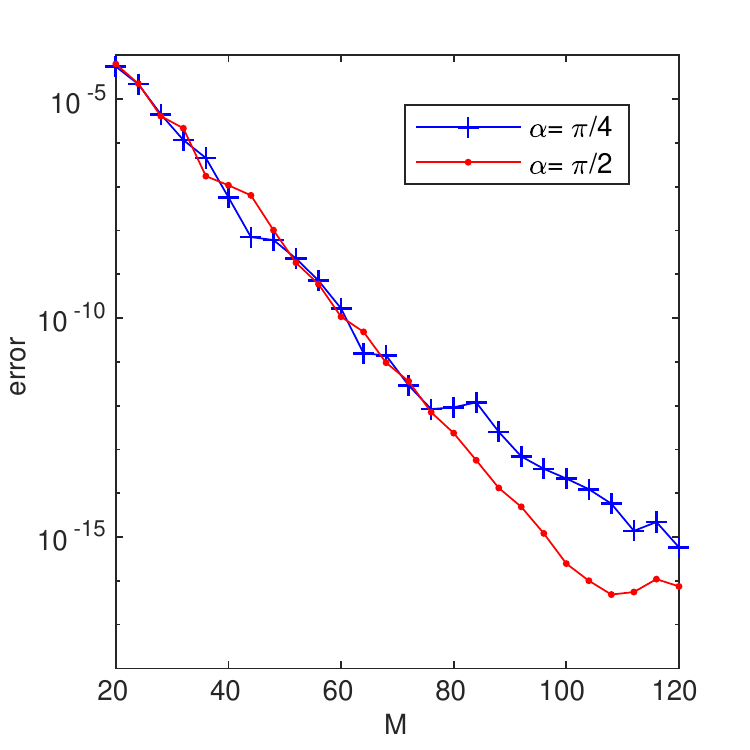}
\ca{Convergence study of the empty box BVP. (Left) shows the collocation points on the walls and the proxy points
(on the expanded unit cell). (Middle) shows the relative $L^{\infty}$ error vs the number of collocation points.  (Right) shows the relative $L^{\infty}$ error vs the number of proxy points.}{f:ebvp_conv}
\efi

\subsection{Evaluation of the periodic Green's function}
\label{s:geval}
  
In this subsection, we address the remaining issue of the evaluation of sums of the following forms at all targets $i=1,\dots,N$, in the case of $N$ large.
\begin{align}
S_{per}(\x_i) &= \sum_{j=1}^N G_{per}(\x_i,\y_j)\f_j \\
S_{per}^p(\x_i) &= \sum_{j=1}^N G_{per}^p(\x_i,\y_j)\cdot \f_j \\
S_{per}^t(\x_i) &= \sum_{j=1}^N G_{per}^t(\x_i,\y_j)\f_j \,.
\end{align}
Here $G_{per}(\x,\y)$ is the generalized perodic Green's function defined in section \ref{subsec:gper}, and
$G_{per}^p(\x,\y)$ and $G_{per}^t(\x,\y)$ are the associated pressure and traction kernels.
$\{\y_j\}_{j=1}^N$ are a given set of point sources lying in the unit cell $\cU$, with given strength $\{\f_j\}_{j=1}^N$, and $\{\x_i\}_{i=1}^{N_t}$ are a given set of target points, also in the unit cell $\cU$.
We now summarize the practical use of the 
decompositions \eqref{eq:GperDecomp} and \eqref{eq:GpperDecomp}.
There are four main steps:
\begin{enumerate} 
\item Compute the ``near part" of the sums at all targets,
\begin{align*}
S_{near}(\x_i) &\defeq \sum_{j=1}^N G_{near}(\x_i,\y_j)\f_j \\
S_{near}^p(\x_i) &\defeq \sum_{j=1}^N G_{near}^p(\x_i,\y_j)\f_j \\
S_{near}^t(\x_i) &\defeq \sum_{j=1}^N G_{near}^t(\x_i,\y_j)\f_j \,.
\end{align*}
This is most efficiently done using a fast algorithm such as a Stokes FMM.
\item Evaluate the discrepancy of the ``near part" sums $S_{near}(\x)$ and $S_{near}^t(\x)$ at the walls, namely:
$\g=[\g_1;\g_2;\g_3;\g_4]$, where
\begin{align*}
\g_1 &= S_{near,R}-S_{near,L} \\
\g_2 &= S_{near,R}^t-S_{near,L}^t  \\
\g_3 &= S_{near,U}-S_{near,D} \\
\g_4 &= S_{near,U}^t-S_{near,D}^t \,.
\end{align*}
Notice that cancellation may be exploited here, so that all evaluations
are distant even when sources approach or lie on the walls.
\item
Use the MFS to solve the empty box BVP with the modified discrepancy $\tilde{\g}=\g_0-\g$, where
$\g_0=[\bm{0}; \F/2|\e_2|; \bm{0}; \F/2|\e_1|]$,
for total force $\F=\sum_{j=1}^N \f_j$, to get the periodizing
coefficients $\{\xi_j\}_{j=1}^M$.
\item
Recover the periodic sums by evaluating the correction pair $(\v,q)$ at the targets via sums \eqref{vrep}--\eqref{qrep}, then adding this to the near results:
\begin{align}
S(\x_i) &= S_{near}(\x_i) + \v(\x_i) \\
S^p(\x_i) &= S_{near}^p(\x_i) + q(\x_i) \\
S^t(\x_i) &= S_{near}^t(\x_i) + \T(\v,q)(\x_i) \label{Stsplit}
\end{align}
\end{enumerate}

\begin{rmk}
It is a direct consequence of Theorem~\ref{t:consistency} that the discrepancy $\tilde{\g}$ obtained from the above procedure automatically satisfies the consistency conditions for the empty box BVP. 
In the language of linear algebra, letting $\d=\{\g(\y_j)\}_{j=1}^M$ and $\tilde{\d}=\{\tilde{\g}(\y_j)\}_{j=1}^M$ be the discretization of $\g$ and $\tilde{\g}$ respectively, we can write $\tilde{\d}=P\d$, where
$P$ is the projection from $\d$ to $\tilde{\d}$.
The statement is that $W^{T}\tilde{\d}\approx 0$ holds true up to the discretization error, i.e. the vector $\tilde{\d}$ always lies in the range of $Q$, even though $\d$ in general does not. 
\end{rmk}

\begin{rmk}[Complexity]
  Since the correction pair $(\v,q)$ is
  smooth inside the unit cell $\cU$, only $M=O(1)$ basis functions are needed, independent of $N$.
  As a result, the linear system solve costs $O(1)$, while
  direct evaluation of the discrepancy and evaluation of the correction
  at the targets both cost $O(N)$.
  The FMM evaluates the ``near part" in $O(N)$ time, so that the total cost remains $O(N)$.
  In practice the total time is dominated by the near-sum FMM from $4N$ sources
  to $N$ targets.
\end{rmk}

\begin{rmk}[Pressure correction for GMRES stability]
Armed with the machinery described in previous sections, we can now apply iterative methods like GMRES to
solve the linear system \eqref{eq:BIEdiscretized}, with one caveat: due to the fact that the generalized
Green's function for the traction $G_{per}^t(\x,\y)$ is determined only up to a constant $c\n_{\x}$, 
a direct implementation of this algorithms results in a stagnation of the GMRES. 

Fortunately this problem
can be fixed easily. One approach is to introduce $\T_{\alpha}(\x) = \T(\x)+\alpha\n_{\x}$, 
where $\T(\x)=\cK[\bsigma]$ is the traction of the periodic single layer potential with density $\bsigma$,
using the notations introduced in earlier sections. Letting $\x_1\in \Gamma_j$ be a fixed point on the
boundary $\Gamma_j$, we pick $\alpha$ so that $\T_{\alpha}(\x_1)\cdot\n_{\x_1}=0$, which implies
$\alpha=-\T(\x_1)\cdot\n_{\x_1}$. We then replace $\T(\x)=\cK[\bsigma]$ in the equation \eqref{eq:BIEPerturbed}
by the modified $\T_{\alpha}(\x)=\T(\x)-\T(\x_1)\cdot\n_{\x_1}$. The rest of the algorithm remains
unchanged. This simple correction successfully stops the GMRES from stagnating.
\end{rmk}

\subsection{Special quadrature for evaluation close to the boundary}
\label{s:close}

When simulating dense suspension of particles in viscous flows, we are often faced with the situations where
particles approach very close to each other. In order to use integral equation methods, we need quadrature
rules for the accurate evaluation of integrals in the form of \eqref{eq:BIEurep}, \eqref{eq:BIEprep}, and \eqref{eq:BIEtrep}, where the target $\x$ can be arbitarily close to the boundary $\Gamma$.

It is well known that a fixed smooth rule leads to the error of $O(1)$ as $\x$ approaches $\Gamma$. 
Fortunately special quadrature rules have been developed for the evaluation of layer potentials with
close to surface targets. Here we recommend the
exterior single-layer rule in \cite{closeglobal}, which builds upon
the ``globally compensated" rule of Helsing--Ojala \cite{helsing_close}.
Both exploit barycentric-type quadratures for Cauchy integrals and their derivatives.

The Stokes single layer potential evaluator is obtained by expressing the single layer potential in terms
of Laplace layer potentials.
\begin{equation}
S[\bsigma](\x) = \frac{1}{2}S_L[\bsigma] + \frac{1}{2}\nabla S_L[\tilde{\sigma}] - \frac{1}{2} x_1\nabla S_L[\sigma_1] -\frac{1}{2} x_2\nabla S_L[\sigma_2]\,,
\end{equation}
where $\bsigma=(\sigma_1,\sigma_2)$, $\tilde{\sigma}(\x)=\x\cdot\bsigma(\x)$.

In our formulation, we also need to evaluate the traction of the Stokes single layer potential, for which
we introduce the following decomposition to complete the story.
\begin{equation}
\begin{split}
S^t[\bsigma](\x) &= x_1\mathcal{H}_x(\mathcal{S}_L(\sigma_1))\cdot \n_x + x_2\mathcal{H}_x(\mathcal{S}_L(\sigma_2))\cdot \n_x \\
&-\mathcal{H}_x(\mathcal{S}_L(\tilde{\sigma}))\cdot \n_x - (\partial_{x_1}\mathcal{S}[\sigma_1]+ \partial_{x_2}\mathcal{S}[\sigma_2])\cdot \n_x, 
\end{split}
\end{equation}
where $\bsigma=(\sigma_1,\sigma_2)$, $\tilde{\sigma}(\x)=\x\cdot\bsigma(\x)$, and $\mathcal{H}_x$ is the Hessian matrix with respect to $\x$.

\begin{rmk} \label{rmk:precomp}
To further reduce the computational cost, we use the special quadrature rule in the following fashion.
We divide the solver into two phases: 
a precomputation phase where a sparse matrix is formed, corresponding to the correction of the quadrature for points that are within $10h$ distance from a boundary, followed by a solution phase,
where GMRES is called to solve the linear system, where in each iteraction an FMM is called and the sparse matrix
is applied.
\end{rmk}

% VVVVVVVVVVVVVVVVVVVVVVVVVVVVVVVVVVVVVVVVVVVVVVVVVVVVVVVVVVVVVVVVVVVVVVVVVVVV
\subsection{Efficient evaluation of the effective viscosity}
\label{s:mueff}

It is important in applications to extract the effective
(homogenized) viscosity, given by \eqref{mueff}
in the case that no particle intersects $D$.
Since physically this formula integrates the horizontal force
transmitted per unit cell, and every object $\Omega_k$ or region of fluid
is in static equilibrium (Reynolds number is zero),
in fact {\em any} curve $\cal C$ may be used
which connects a point on $L$ to its corresponding
periodic image on $R$ yet does not touch any objects.
That is, $\int_D$ in \eqref{mueff} may be replaced by $\int_{\cal C}$;
the proof is via Gauss' Law \eqref{eq:SLGaussLaw1StokesForce}.
Thus there is invariance to ``deformation'' of the integration contour,
as in complex analysis.

Since in a large scale simulation with many particles it is common for
at least one particle to be intersecting $B$, integration on $D$ cannot be used.
However, in this case of many particles, finding a $\cal C$ as above,
and choosing a numerical
quadrature scheme sufficiently resolved to handle the
rapid traction changes as $\cal C$ passes between nearly touching objects,
would be cumbersome
Thus, in the spirit of \cite[Sec.~2.6]{ahb},
we propose a much more efficient and (we believe) elegant
contour deformation method that exploits
the split of the periodic BIE representation into near sum and correction parts.
This method only involves {\em distant} interactions,
thus allows $O(1)$ quadrature nodes independent of the complexity of the
geometry.

We now present a method to evaluate the force integral
$\int_{\cal C} \T(\u,p) ds_\x$, where $\cal C$ is a deformation of $D$,
needed for \eqref{mueff}.
Without loss of generality, we assume that the unit cell $\cU$ is centered at the origin.
We consider $K=1$ (a single object in the unit cell); the generalization
to $K>1$ is straightforward.
If this object $\Omega$
intersects the wall $D$, we assume that $\cal C$ has been
deformed to pass {\em below} $\Omega$, and that its endpoints are the
same as those of $D$.
We split the boundary of $\Omega$ as $\Gamma = \cup_{i=1}^4 \Gamma_i$,
where $\Gamma_i$ is the part of $\Gamma$
that lies in the $i$th ``quadrant'' (in the sense of the skew unit cell).
See Fig \ref{f:deform_contour}.
Define the (possibly disconnected) curve
$\tilde{\Gamma}= \cup_{i=1}^4 \tilde{\Gamma}_i$, 
where $\tilde{\Gamma}_1=\Gamma_1-\e_1-\e_2$, $\tilde{\Gamma}_2=\Gamma_1-\e_2$, 
$\tilde{\Gamma}_3=\Gamma_3$, and $\tilde{\Gamma}_4=\Gamma_1-\e_1$.

%fffffffffffffffffffffff
\bfi  % fffffffffffffffffffff
\ig{width=1.95in}{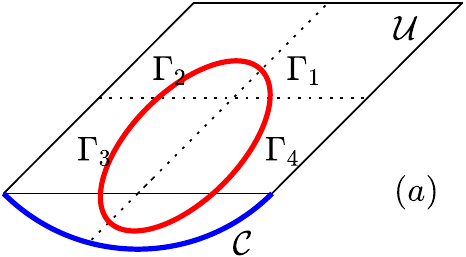}
\ig{width=1.95in}{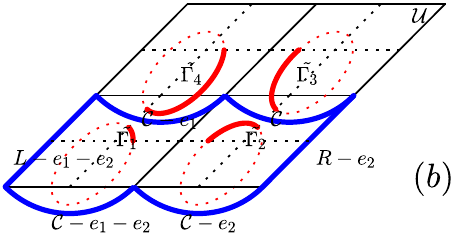}
\ig{width=1.95in}{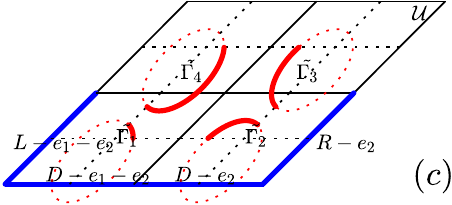}
\ca{Illustration of the deformation of contour technique for the evaluation of effective viscosity. (a) shows the unit cell with one object that intersects the bottom wall. The integration on the bottom wall can be replaced by that on the curve $\cal C$. (b) shows the shifted pieces of the boundary, and the shifted copies of $\cal C$. Gauss' Law is applied on the closed contour in bold blue, which encloses $\tilde{\Gamma}_1$ and $\tilde{\Gamma}_2$. (c) Cancelling the terms involving shifted copies of $\cal C$, we obtain an expression that involves sources on the shifted pieces of the boundary (in bold red) and targets on the shifted walls (in bold blue) only. The sources and targets are well separated.}{f:deform_contour} 
\efi

\begin{pro}
  Let $\T=S_{\Gamma}^{per,t}[\bsigma]=S_{\Gamma}^{near,t}[\bsigma]+ \T(\v,q)$
  be the traction on $\cal C$
  of a periodic single layer potential that is decomposed into a near direct image sum plus a correction part given by velocity field $\v$ and pressure $q$,
  as in \eqref{Stsplit}.
  Define $\tilde{\bsigma}$ on $\tilde{\Gamma}$ in the natural way so
  that $\tilde{\bsigma}|_{\tilde{\Gamma_i}}=\bsigma|_{\Gamma_i}$.
  Then,
\begin{equation} \label{eq:contourdef}
\begin{split}
\int_{\cal C}\T\,ds_{\x} =& \int_{R-\e_2} S_{\tilde{\Gamma}}^t[\tilde{\bsigma}] ds_{\x} -\int_{L-\e_1-\e_2} S_{\tilde{\Gamma}}^t[\tilde{\bsigma}] ds_{\x} + 2\int_{D-\e_1-\e_2} S_{\tilde{\Gamma}}^t[\tilde{\bsigma}] ds_{\x} \\
& +2\int_{D-\e_2} S_{\tilde{\Gamma}}^t[\tilde{\bsigma}] ds_{\x} 
-\int_{\Gamma_1\cup\Gamma_2} \bsigma ds_{\x} + \int_{D} \T(\v,q)
\end{split}
\end{equation}
\end{pro}

\begin{proof}
Since the free space kernel $G^t(\x,\y)=G^t(\x-\y)$ is translation invariant, it is straightforward to
verify that 
\begin{equation}
S_{\Gamma}^{near,t}[\bsigma] (\x) = S_{\tilde{\Gamma}}^t[\tilde{\bsigma}](\x)+
S_{\tilde{\Gamma}}^t[\tilde{\bsigma}](\x-\e_1)+
S_{\tilde{\Gamma}}^t[\tilde{\bsigma}](\x-\e_2)+
S_{\tilde{\Gamma}}^t[\tilde{\bsigma}](\x-\e_1-\e_2)\,,
\end{equation}
simply by fixing the sources to be $\tilde{\Gamma}$ and shifting
the targets accordingly. Integrating
both sides on $\cal C$ leads to 
\begin{equation} \label{eq:targshift}
\int_{\cal C}S_{\Gamma}^{near,t}[\bsigma] ds_{\x} =\int_{D} S_{\tilde{\Gamma}}^t[\tilde{\bsigma}] ds_{\x}+
\int_{{\cal C}-\e_1} S_{\tilde{\Gamma}}^t[\tilde{\bsigma}] ds_{\x}+
\int_{{\cal C}-\e_2} S_{\tilde{\Gamma}}^t[\tilde{\bsigma}] ds_{\x}+
\int_{{\cal C}-\e_1-\e_2} S_{\tilde{\Gamma}}^t[\tilde{\bsigma}] ds_{\x}\,,
\end{equation}
However, \eqref{eq:SLGaussLaw1StokesForce} implies
\begin{equation}
\int_{\cal B} S_{\tilde{\Gamma}}^t[\tilde{\bsigma}] ds_{\x} = \int_{\tilde{\Gamma_1}\cup\tilde{\Gamma_2}}
\tilde{\bsigma} ds_{\x}\, = \int_{\Gamma_1\cup \Gamma_2} \bsigma ds_{\x}\,, 
\label{encl}
\end{equation}
where $\cal B$ is the closed loop given by the six curves
$L-\e_1-\e_2$, ${\cal C}-\e_1-\e_2$, ${\cal C}-\e_2$, $R-\e_2$,
${\cal C}$, and ${\cal C}-\e_1$,
oriented in the counterclockwise sense.
Note that $\cal B$ is the boundary of the two (deformed) unit cells 
$\cU-\e_1-\e_2 \;\cup\; \cU -\e_2$,
and encloses the sources $\tilde{\Gamma}_1$ and $\tilde{\Gamma}_2$, giving
\eqref{encl}.
See Fig. \ref{f:deform_contour}.
We now add \eqref{eq:targshift} and \eqref{encl}, and observe
cancellations of the wall terms ${\cal C}$ and ${\cal C}-\e_1$,
leaving \eqref{eq:contourdef} except with $\cal C$ in place of $D$ on the
right-hand side.
Finally, since they are around half a unit cell from any sources,
${\cal C}$ and ${\cal C}-\e_1$ may be deformed back to $D$ and $D-\e_1$
respectively without any effect.
$\cal C$ may similarly be replaced by $D$ in the
final correction term of \eqref{eq:contourdef} because $(\v,q)$ is a
Stokes pair throughout a neighborhood of $\cU$.
\end{proof}

\begin{rmk}
  The formula \eqref{eq:contourdef} involves integration in the far field of source curves only (all distances are at least half a unit cell), and the term involving $\v$ is also smooth.
  As a consequence, a smooth
quadrature rule with $O(1)$ nodes is enough to compute it accurately.
In addition, a little bookkeeping shows that the values of $S_{\tilde{\Gamma}}[\tilde{\bsigma}]$ on the four shifted walls
$R-\e_2$, $R-\e_1-\e_2$, $L-\e_1-\e_2$, and $D-\e_1-\e_2$ also appear when forming the discrepancy vector.
Consequently no extra cost is needed in evaluating $S_{\tilde{\Gamma}}[\tilde{\bsigma}]$ at the quadrature nodes.
They are readily available as a byproduct of the periodization by the MFS. 
\end{rmk}

% ttttttttttttttttttttttttttttttttttttttttttttttttttttttttttttttttttttttttttttt
\section{Time evolution and time stepping}
\label{s:step}

So far we have presented the solution of the quasi-static problem
for a given unit cell and particle geometry.
Recall that the quasi static solver takes as input the viscosity $\mu$,
the shear rate $\gamma$, the current lattice vectors $\e_1$ and $\e_2$, and the current boundaries of the rigid particles $\{\Gamma_j\}_{j=1}^{N_o}$, and returns the linear velocities of the particles $\{\v_j\}_{j=1}^{N_o}$, their angular velocities
$\{\omega_j\}_{j=1}^{N_o}$, and the velocity field $\u(\x)$ and pressure $p(\x)$ for $\x\in\cU\backslash \overline{\Omega_\Lambda}$.
However, the effects of the evolving particle positions and angles,
and the resulting changes in viscosity, are more of interest.

We assume that the background flow is shearing in the $y$-direction
at a constant rate $\gamma$, so have $\e_1(t)=(1,0)$ and $\e_2(t)=(\gamma t - \mbox{round}(\gamma t), 1)$,
where round$(x)$ is defined as the integer nearest to $x\in\RR$.
Here the rounding operation causes $\e_2(t)$ to jump backwards,
once per shear time, in a way that remains consistent with the periodicity
of the lattice of particles. This ensures a lattice whose skewness remains
bounded.
(More general prescribed time-dependent shear of a general unit cell is of course possible, as long as the skewness is kept small in a similar fashion.)
Let $\{\x_j^c(t)\}_{j=1}^{N_o}$ and $\{\theta_j(t)\}_{j=1}^{N_o}$ be the centers
and angles of the particles, whose
initial configuration was $\{\Gamma_j(0)\}_{j=1}^{N_o}$.
Since particle shapes remain unchanged in time, there exists a fixed function $\X_b$
that maps the centers and angles to the boundaries: $\Gamma_j=\X_b(\x_j^c,\theta_j)$. 
We further assume that the function $\X_b$ is given explicitly and at $t=0$ the initial values of the centers
and angles are given by $\x_{j,0}^c$ and $\theta_{j,0}^c$.

We stack the centers and angles into one vector $\s=(\x_1^c, \cdots, x_{N_o}^c,\theta_1,\cdots,\theta_{N_o})$, and
observe that $\v_j=\frac{d \x_j^c}{dt}$, and $\omega_j=\frac{d \theta_j}{dt}$. Using these notations, we can reinerpret the quasi-static solver as a function $F$ that maps $(t,\s)$ to the velocity $\frac{d\s}{dt}$, leading
to the following first-order ODE system
that governs the evolution:
\begin{equation} \label{eq:ode_mob}
\begin{cases}
\frac{d\s}{dt} & = F(t,\s) \\
\s(0) &= \s_0
\end{cases}
\end{equation}
In our case of constant shear rate, this ODE is in fact autonomous.

Many numerical methods for ODEs can be applied, the simplest one being forward Euler, which leads to $s^{(n+1)}=s^{(n)}+\Delta t F(t_n,s^{(n)})$, where $s^{(n)}:=s(t_n)$.
Higher order and/or adaptive methods can also
be applied, but with one caveat: when objects become
too close to each other (less than $\bigO(h^2)$), accuracy is lost in the
quasi-static solve, which we observe can cause stagnation
(here we tested the explicit adaptive RK4 ODE solver {\tt ode45} in MATLAB).
In practice it is common to 
add non-hydrodynamic short-range repulsion forces,
or non-overlapping constraints \cite{yan2019scalable},
in such simulations.
However, that is beyond the scope of this paper.
Instead we dedicate our attention to the quasi-static solver and adopt forward Euler for time stepping.

% NNNNNNNNNNNNNNNNNNNNNNNNNNNNNNNNNNNNNNNNNNNNNNNNNNNNNNNNNNNNNNNNNNNNNNNNNNNNN
\section{Numerical examples}
\label{s:num}

{\bf Example 1.} As a first example, we carry out a convergence study for the quasi-static problem.
We define the boundaries to be two ellipses $\Gamma_1: \{(x_1+0.25\cos{\theta}, 0.125\sin{\theta}),\, \theta\in
[0,2\pi]\}$, and $\Gamma_2: \{(x_2+0.125\cos{\theta}, 0.25\sin{\theta}),\, \theta\in[0,2\pi]\}$. $x_1$ and $x_2$
are chosen so that $\Gamma_1$ and $\Gamma_2$ are distance $d$ apart, i.e. $|x_1-x_2|=d+0.375$. It is well known
that as $d$ shrinks, the traction becomes more and more sharply peaked, requiring more discretization points
on the boundary. In fact it is pointed out in \cite{sanganimo,wu2019arxiv} that the width of the ``bump" scales as $O(\sqrt{d/\kappa})$, where $\kappa=\kappa_1+\kappa_2$ is the sum of the curvatures of the boundaries at the close to touching point.

In this example, we have carried out experiments for $d=10^{-1},\,10^{-3},\,10^{-5}$ respectively, each for
different numbers of discretization points on the boundaries. Results are shown in Figure \ref{f:stat_conv} where we measure the
resolution by the residuals from recovering the rigid body motion as well as the errors in the linear and
angular velocities. In all the experiments, we set the tolerance of the FMM to be $10^{-12}$ and the tolerance
for the GMRES to be $10^{-10}$. The reference solution is computed with $10^4$ points on the boundary.

A spectral convergence is observed in each case. The onset of the convergence shows a $O(1/\sqrt{d})$ scaling with
respect to the distance, consistent with the asymptotic analysis. In the $d=10^{-1}$ case, it requires around $60$ points to achieve an error of $10^{-12}$, while in the $d=10^{-5}$ case, it requires around $6000$ points to get an error of $10^{-8}$, demonstrating the robustness of our method for very close to touching boundaries.

\bc
\bfi  % fffffffffffffffffffff
\centering\ig{width=3in}{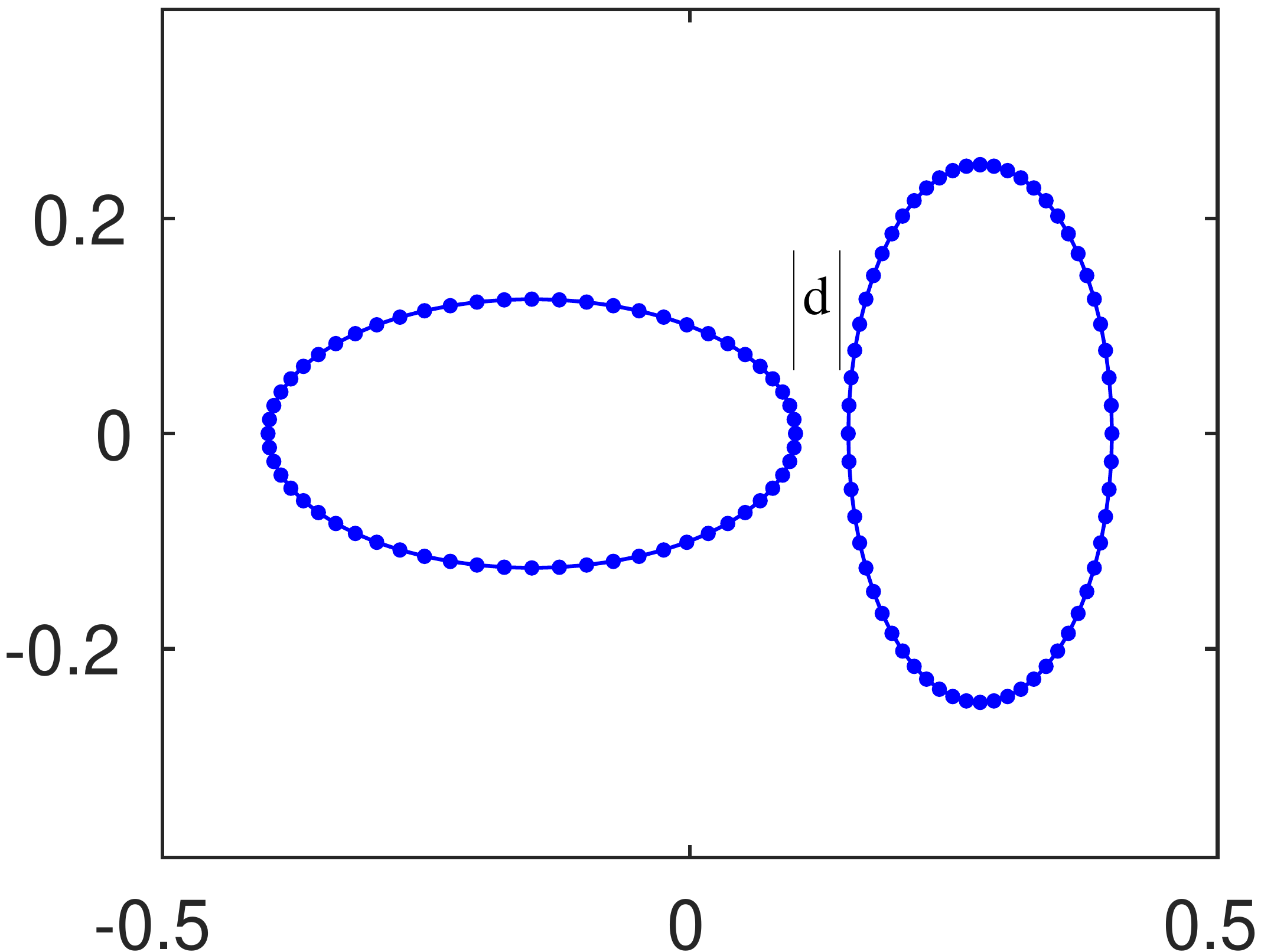}
\ca{Geometry for Example 1: two ellipses with aspect ratio 2 that are a small
  distance $d$ apart.}{f:geom_stat_conv}
\efi

\ec

\bfi  % fffffffffffffffffffff
\ig{width=1.95in}{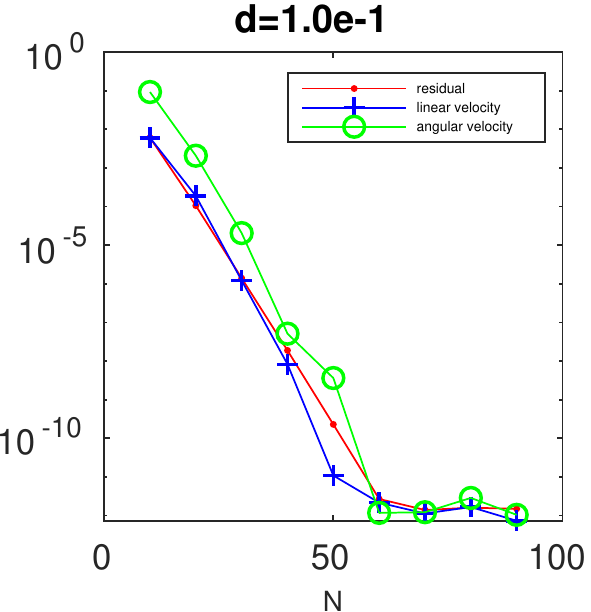}
\ig{width=1.95in}{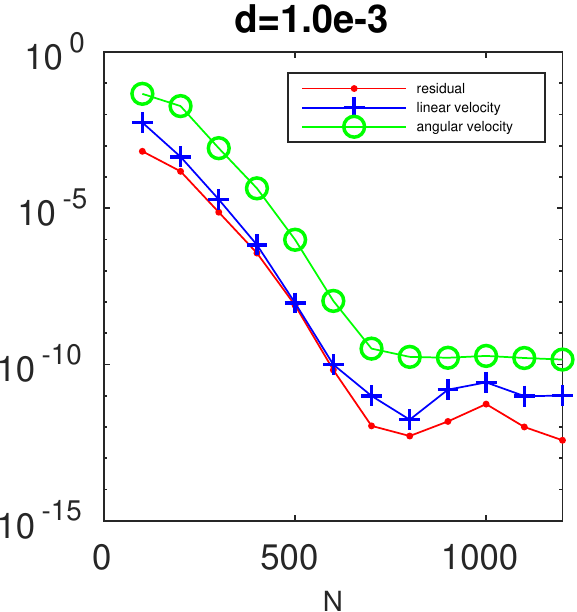}
\ig{width=1.95in}{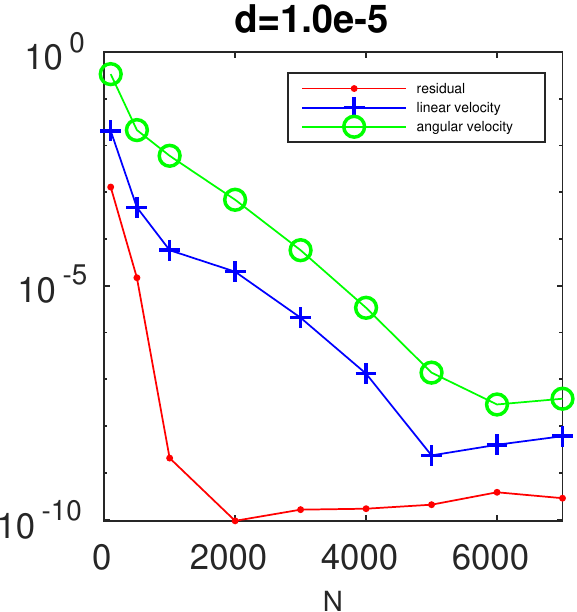}
\ca{Convergence study for the quasi static BVP, where the boundaries are two ellipses placed at a certain distance apart (Example 1). In all three plots, $N$ is the number of points on each ellipse. A spectral convergence in $N$ is observed in each case, while the onset $N$ of the convergence is seen to scale as $O(1/\sqrt{d})$.}{f:stat_conv}
\efi

{\bf Example 2.} In this example, we demonstrate the efficiency of our solver for the quasi-static BVP with complex geometry. We create a large number $K$ of random geometries where each boundary takes the form
$r(\theta)=s(1+a\cos{\omega\theta+\phi})$, with $\phi$ random, uniformly in $[0,0.5]$, $\omega$ randomly chosen from $\{2,3,4,5\}$, and $s$ varying over a size ratio of 4. They are then scaled and shifted to fill out the unit
square, where overlapping ones are discarded. The geometries with $K=100$ and $K=1000$ are shown in Figure
\ref{f:stat_islands}. 

The quasi-static BVP with viscosity $\mu=0.7$ and shear rate $\gamma=1$ is then solved for $K=100, 200, \cdots 1000$. In each case, we put down $N_k=350$ discretization points on each boundary,
and fix the FMM tolerance to be $10^{-9}$ and the tolerance of the residual to be $10^{-8}$.
As is pointed out in remark \ref{rmk:precomp},
the solver consists of a precomputation phase and a solution phase.
The CPU time consumed in each phase and the number of iteractions are given in Figure \ref{f:stat_islands_performance}. We make the observation that the time per iteration and the time of precomputation have a
clean linear growth with respect to the complexity of the boundary, and the number of iteractions has a
mild growth. The overall complexity is roughly linear.

\bfi  % fffffffffffffffffffff
\ig{width=3in}{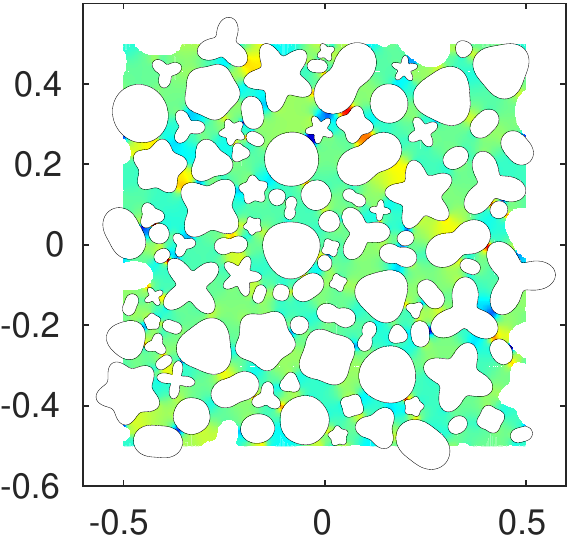}
\ig{width=3in}{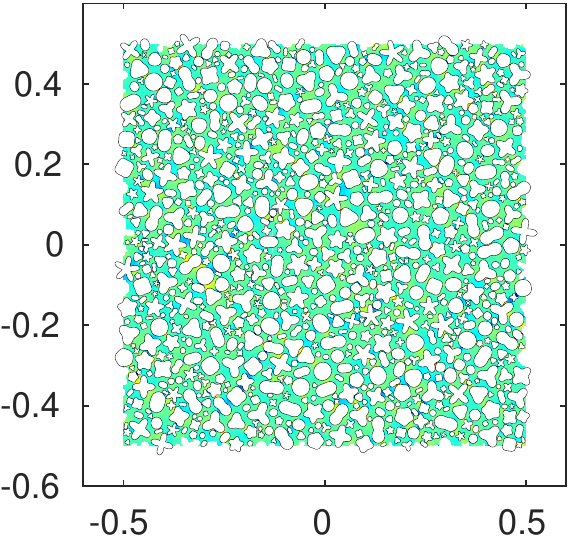}
\ca{Pressure field of the quasi-static BVP with complex boundaries (Example 2). In the left plot, the boundary consists of 100 closed curves and in the right plot, there are 1000.}{f:stat_islands}
\efi

\bfi  % fffffffffffffffffffff
\ig{width=1.95in}{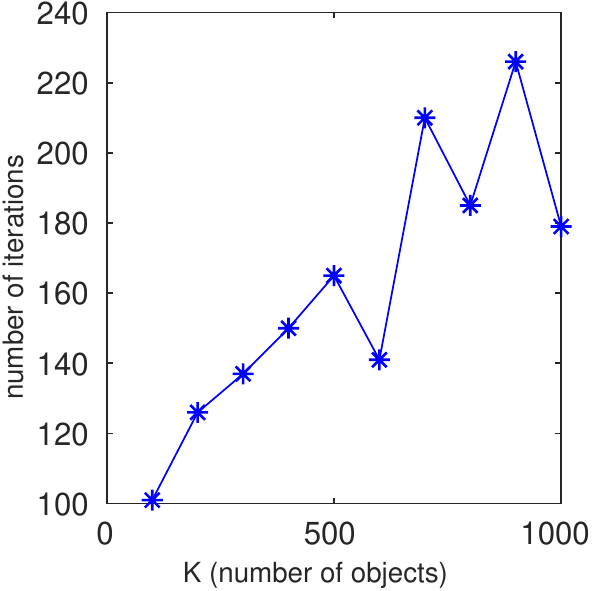}
\ig{width=1.95in}{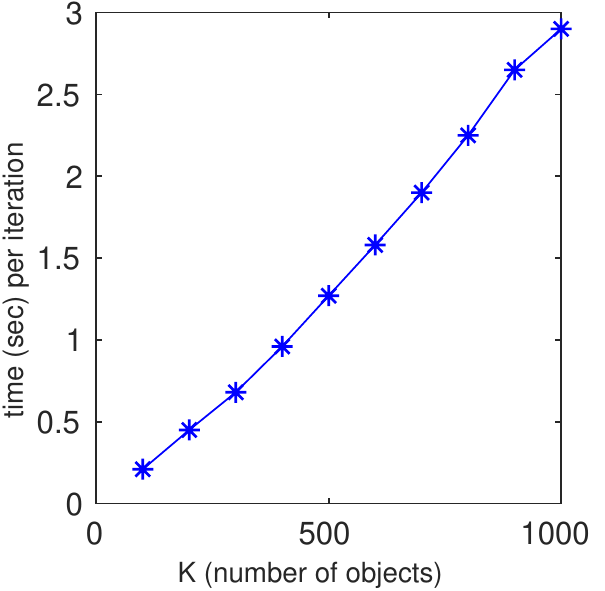}
\ig{width=1.95in}{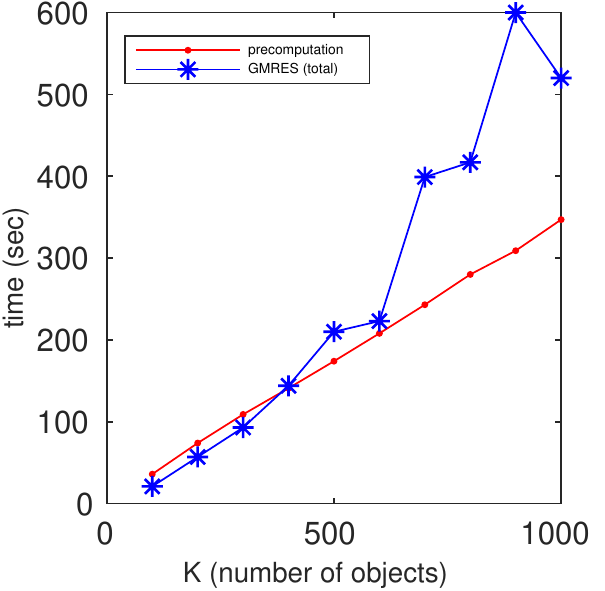}
\ca{Scaling of the number of iterations, the CPU time consumed in each iteration, the total GMRES time, and the precomputation time with respect to the complexity of the geometry.}{f:stat_islands_performance}
\efi

{\bf Example 3.} In this example, we study the evolution of the effective viscosity over a long time, for which
we combine our solver for the quasi-static BVP and the Forward Euler time stepping method for the resulting ODE
system as in equation \ref{eq:ode_mob}. 

The initial configuration is chosen to be $25$ ellipses with aspect ratios in $[1,2]$, centered on a uniform
$5\times 5$ grid on the unit square, as is illustrated in Figure. The volume fraction is $0.32$. We fix viscosity $\mu=1$ and shear rate $\gamma=1$, and discretize each ellipse with $200$ points. Solutions of $\u(\x,t)$ and $p(\x,t)$ are computed on a $200\times 200$
grid on the unit cell, as well as the linear and angular velocities of the ellipses $\{(\v_i(t),\omega_i(t))\}_{i=1}^{25}$, and the effective viscosity $\mu_{\tbox{eff}}$ at $t\in[0,50]$.
Snapshots of the solution at different times are shown in Figure \ref{f:ellipses_snapshots}, and the effective viscosity is shown in Figure \ref{f:ellipses_mueff},
which suggests that the system has reached a stochastic steady state.

We validate the convergence in time by carrying out a self convergence study in $\mu_{\tbox{eff}}(t)$ over time
intervals $[0,1]$ and $[10,11]$. The results are given in table \ref{t:ellipses_conv}, which shows that first order
convergence is always achieved for $t\in[0,1]$, while for $t\in[10,11]$, a much smaller $\Delta t$ is required for
the expected order to be observed.
This growth in prefactor in the convergence is to be expected if
the ODE system increasingly amplifies deviations in initial conditions with time, for instance if there is a positive Lyapunov exponent.

\bc
\bfi  % fffffffffffffffffffff
\mbox{
\ig{width=2in}{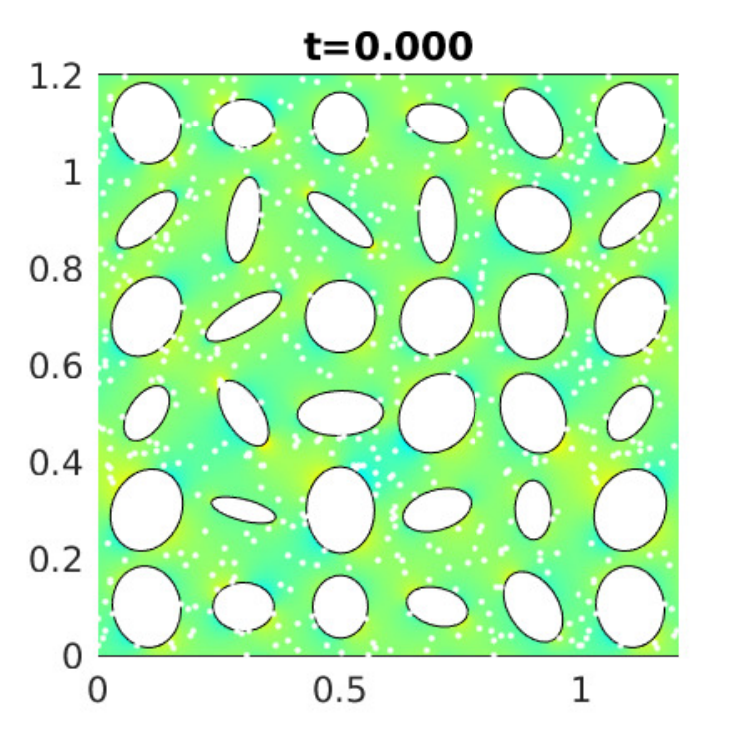}
\ig{width=2in}{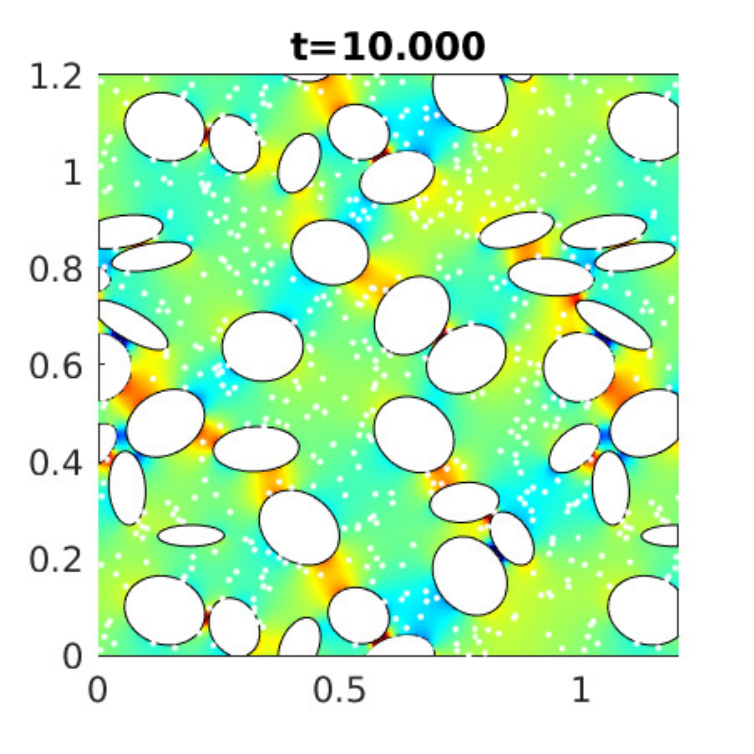}
\ig{width=2in}{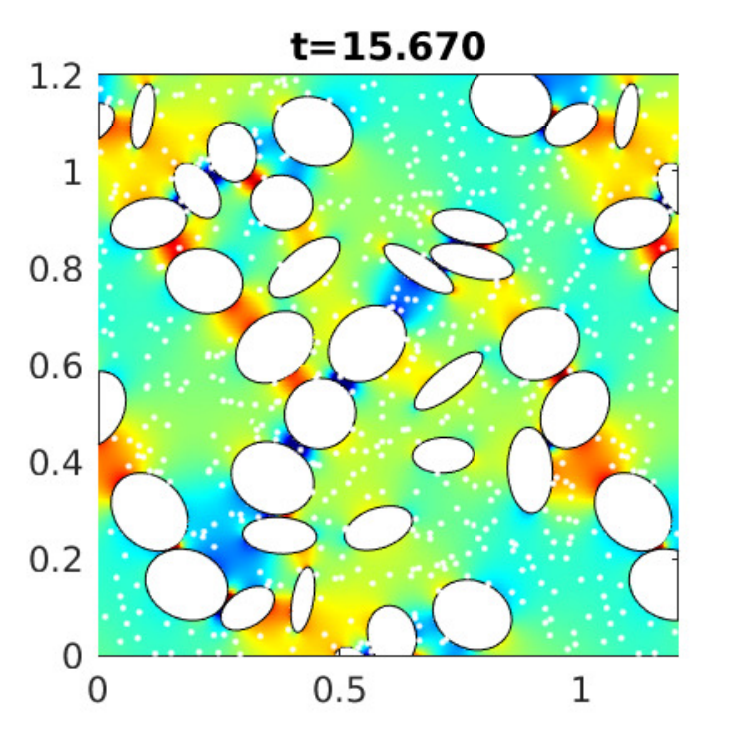}
}
\ca{Snapshots of the solution at different times for 25 ellipses (Example 3). The pressure field is shown in color, and the positions of the boundaries and extra tracer points (white dots) are plotted.
  The last snapshot is at $t=15.67$ where $\mu_\tbox{eff}(t)$ is very close to
  a peak (see Fig.~\ref{f:ellipses_mueff}); notice the presence of diagonal
  force chains (compressional high-pressure lines
  pointing at around 4 o'clock, extensional low-pressure lines at around
  1 o'clock).}{f:ellipses_snapshots}
\efi
\ec

\bfi  % fffffffffffffffffffff
\centering\ig{width=3in}{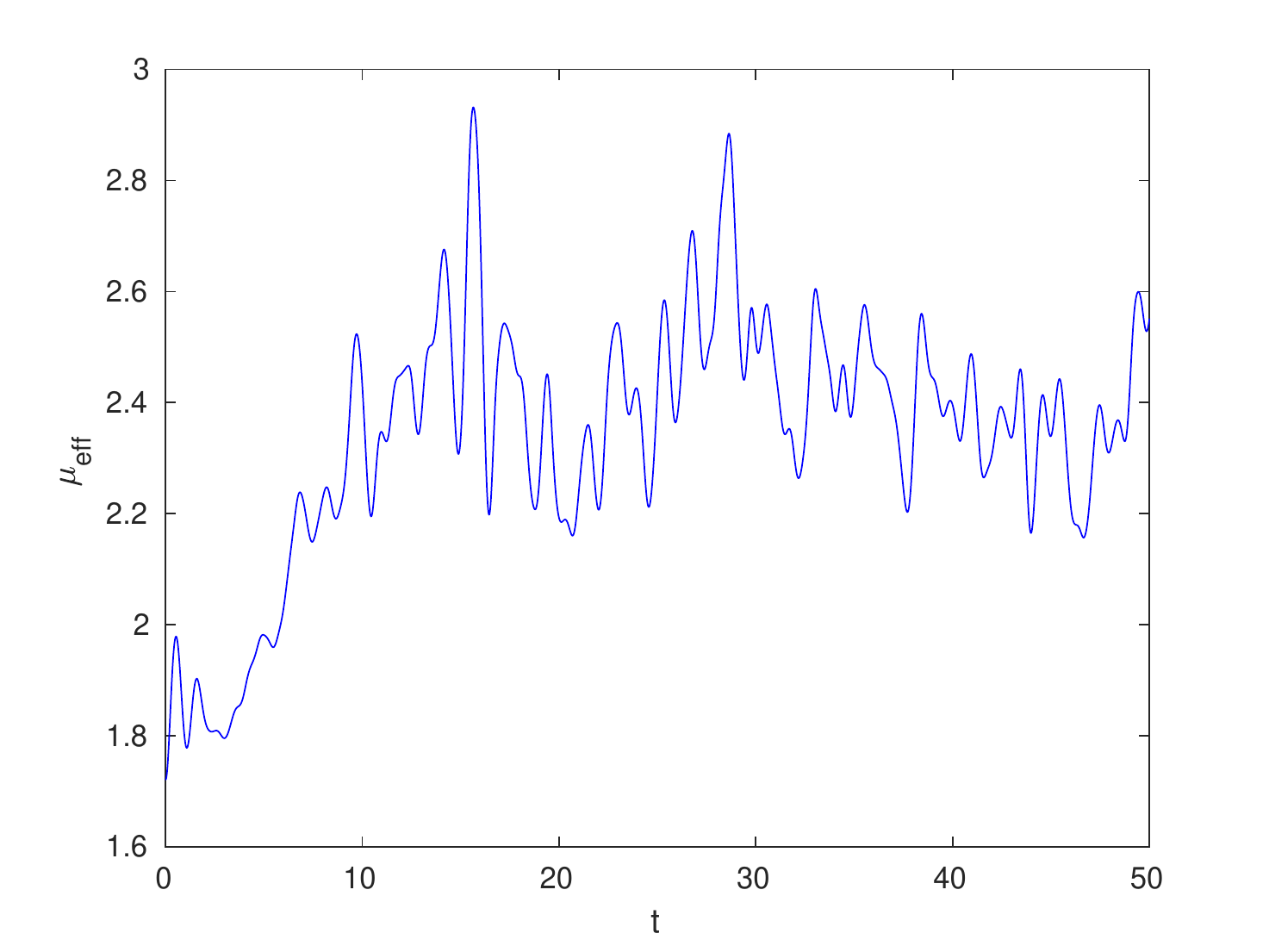}
\ca{Time-evolution of the effective viscosity for the 25 ellipses example.}{f:ellipses_mueff}
\efi

\begin{table} %tttttttttttttt
\begin{tabular}{llcllcl} 
$\Delta t$ & $\mu(1)$ & $||\mu_{\Delta t}-\mu_{2\Delta t}||_{L^2([0,1])}$ & order &
$\mu(11)$ & $||\mu_{\Delta t}-\mu_{2\Delta t}||_{L^2([10,11])}$ & order \\
\hline
1.000e-2 & 1.788484 & N/A     & N/A  & 2.345831 & N/A     & N/A \\
5.000e-3 & 1.788799 & 1.43e-3 & N/A  & 2.336161 & 7.43e-2 & N/A \\
2.500e-3 & 1.788956 & 7.11e-4 & 1.00 & 2.331546 & 5.83e-2 & 0.35 \\
1.250e-3 & 1.789035 & 3.55e-4 & 1.00 & 2.329596 & 3.65e-2 & 0.68 \\
6.250e-4 & 1.789074 & 1.77e-4 & 1.00 & 2.328734 & 2.04e-2 & 0.84 \\
3.125e-4 & 1.789093 & 8.90e-5 & 0.99 & 2.328325 & 1.08e-2 & 0.92 \\
\hline
\end{tabular}
\ca{Self convergence study of $\mu_{\tbox{eff}}(t)$. Two regimes are tested: at the beginning when $t\in[0,1]$, and when
the system enters equilibrium around $t\in[10,11]$. Function values at $t=1$ and $t=11$ are given, as well as
the $L^2$ error over these time intervals, and the empirical order of convergence. This shows $\bigO(\Delta t)$ convergence even at later times.}{t:ellipses_conv}
\end{table}

{\bf Example 4.} We continue the study of the time evolution of the effective viscosity, with a small number of more complicated shapes.
In this last example, we define the boundary to be a star shape, which, as in example 2, takes the form
$\x(\theta)=(r(\theta)\cos{\theta}, r(\theta)\sin{\theta})$, where $\theta\in [0,2\pi]$ and $r(\theta)=0.25(1+0.5\cos{\theta})$. 
We then place two such stars in the unit square, one centered at $(0.25,0.25)$ and the other centered
at $(0.75,0.75)$; see Figure~\ref{f:stars}. As in example 3, we fix the background viscosity $\mu=1$ and shear rate $\gamma=1$, and use
forward Euler as the time stepping method, with $\Delta t=5.0e-4$. Each boundary is discretized using $2500$ points.

Snapshots of the solution and the effective viscosity are
shown in Figure \ref{f:stars}.
Approaching $t=0.46$, $\mu_\tbox{eff}(t)$
grows rapidly without limit, apparently heading towards an infinite value
within finite time. This is associated with the geometry becoming
``jammed'', ie, approaching a geometry
where the rigid objects {\em themselves} prevent further shearing.
One might question whether this blow-up is a numerically credible;
however,
we have tested convergence of the spatial solve and find that
errors are below $10^{-7}$ for $\mu_\tbox{eff} < 10^5$, ie, for all points in
the bottom left panel of Figure \ref{f:stars}.

\begin{rmk}[Lubrication theory]
  It is possible to use a thin-film Stokes model to predict an asymptotic
  blow-up of the form
  \be
  \mu_\tbox{eff}(t) \sim c |t-t^*|^{-\beta},
  \label{blowup}
  \ee
  where $t^*$ is the time where objects touch.
  As Figure \ref{f:stars} shows, the dominant cause of force
  is high pressure within a ``reservoir'' (red) whose volume is contracting
  approximately linearly in time. The resulting constant flow rate $Q$ must
  exit through two narrow channels whose width shrinks like
  $d \propto |t-t^*|$.
  Taking local coordinates where $x$ is location along the narrow channel,
  the Reynolds equation for channel flow \cite[Sec.~22.1]{pantonbook}
  is $\partial p/\partial x =-12Q\mu/h(x)^3$,
  where $h(x)$ is the local
  width.
  (See bottom right panel of Fig.~\ref{f:stars}.)
  For a smooth curves, $h(x) = d + cx^2 + O(x^4)$ holds for small $x$,
  so that $\partial p/\partial x$ is a ``bump'' function of height
  $O(d^{-3})$ and width $O(d^{1/2})$.
  Thus integrating $p$ on $(-\infty,\infty)$ gives a pressure drop
  $O(d^{-5/2})$, giving $\beta=5/2$ in \eqref{blowup}.
\label{r:lub}\end{rmk}

We take the values of $\mu_{\tbox{eff}}(t)$ for $t\in[0.4,0.44]$, where the spatial solve is accurate and $\mu_{\tbox{eff}}(t)$ is in an asymptotic regime, and apply nonlinear least squares to fit the parameters $c$, $t^*$ and $\beta$. The results are $c=2.187$, $t^{*}=0.458$ and $\beta=2.431$, with a relative residual of $0.0004$ and excellent agreement visible in Figure \ref{f:stars}.
The $\beta$ value is very close to the predicted $2.5$.

\begin{rmk}[Single-contact jamming]
  The $\mu_\tbox{eff}$ blow-up observed above relied on a reservoir
  trapped by two contact points. Another type of jamming is possible,
  with a single approaching contact point between two smooth surfaces.
  This situation is known as viscous adhesion \cite[Sec.~22.5]{pantonbook}.
  By contrast, now the flow depends on location, $Q(x) \propto x$,
  and a similar asymptotic analysis as above gives
  a pressure peak of $O(d^{-2})$, width $O(d^{1/2})$, thus
  the weaker power $\beta=3/2$. (See bottom right panel of Fig.~\ref{f:stars2}.)
\end{rmk}

{\bf Example 5.}
An initial configuration giving single-contact jamming is
as follows. The two star shapes are $\x(\theta)=(r(\theta)\cos{\theta}, r(\theta)\sin{\theta})$, where $\theta\in [0,2\pi]$ and $r(\theta)=0.275(1+0.5\cos{\theta})$. 
We then shift them so that the first one is centered at $(0.25,0.25)$ and the second at $(0.75,0.75)$. The second star is then rotated counter-clockwise by $\pi/4$. We keep other parameters unchanged and run the same simulation.

As Figure \ref{f:stars2} shows,
jamming is observed for $t\approx 0.47$. We take the values of $\mu_{\tbox{eff}}(t)$ for $t\in[0.4,0.45]$, and as before fit the parameters in \eqref{blowup},
getting $c=12.16$, $t^*=0.472$,
and $\beta=1.31$, with a relative residual of $0.006$.
The power law does not match the predicted $\beta=1.5$ as well as
in Example 4, but is still quite close.

\bfi  % fffffffffffffffffffff
\ig{width=1.95in}{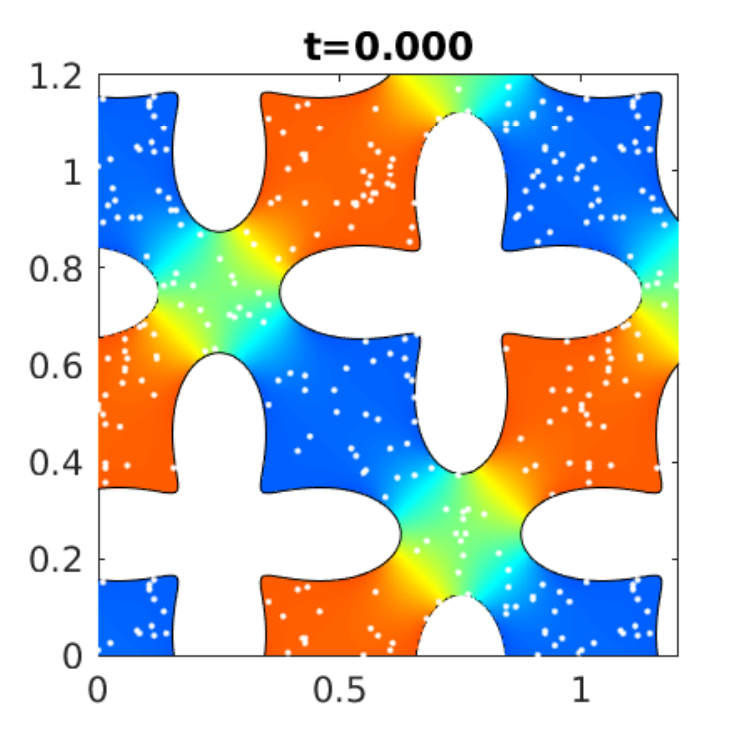}
\ig{width=1.95in}{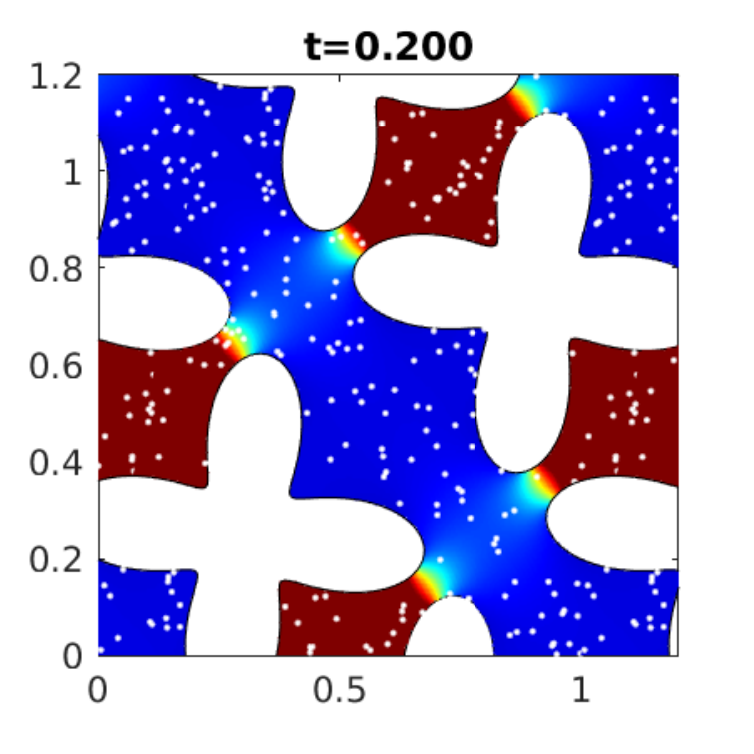}
\ig{width=1.95in}{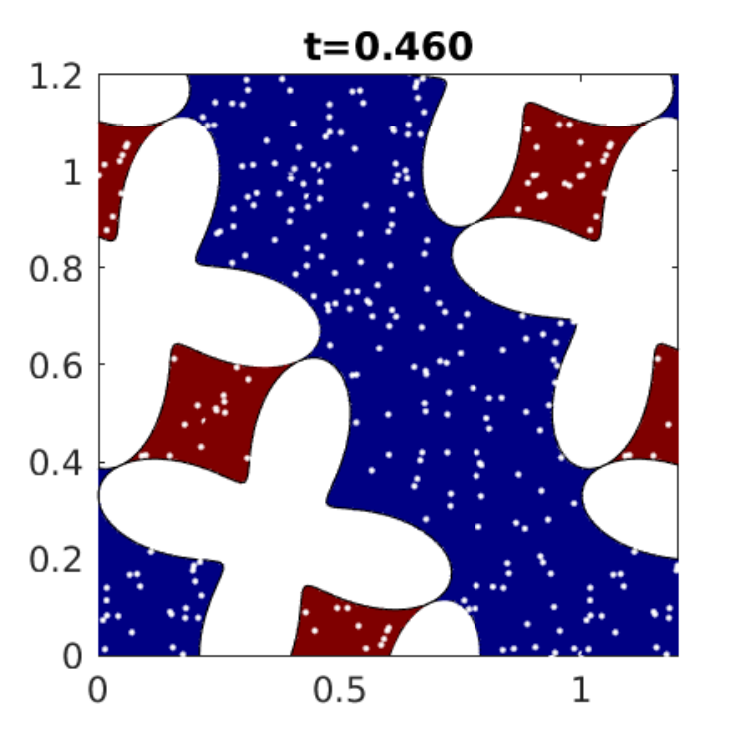}
\\
\mbox{
  \ig{width=1.8in}{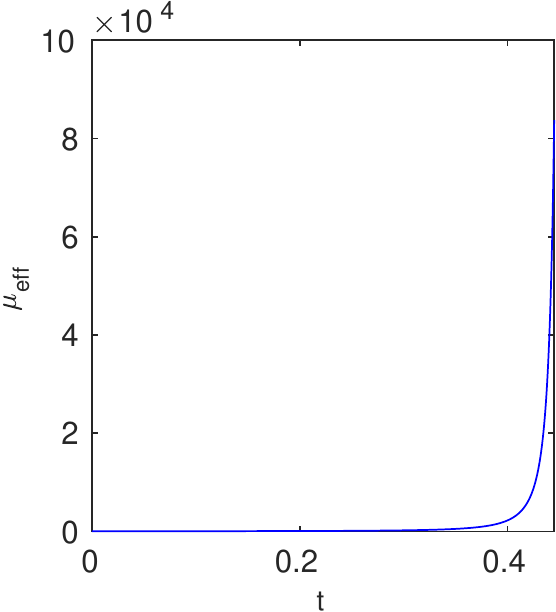}
  \;
  \ig{width=1.8in}{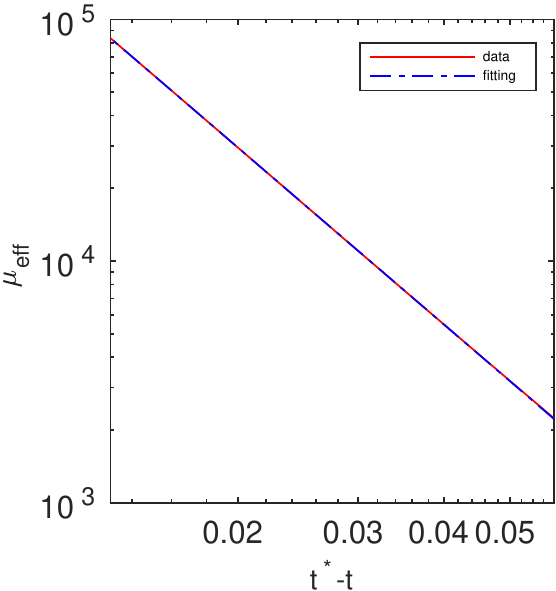}
  \;
  \ig{width=2.2in}{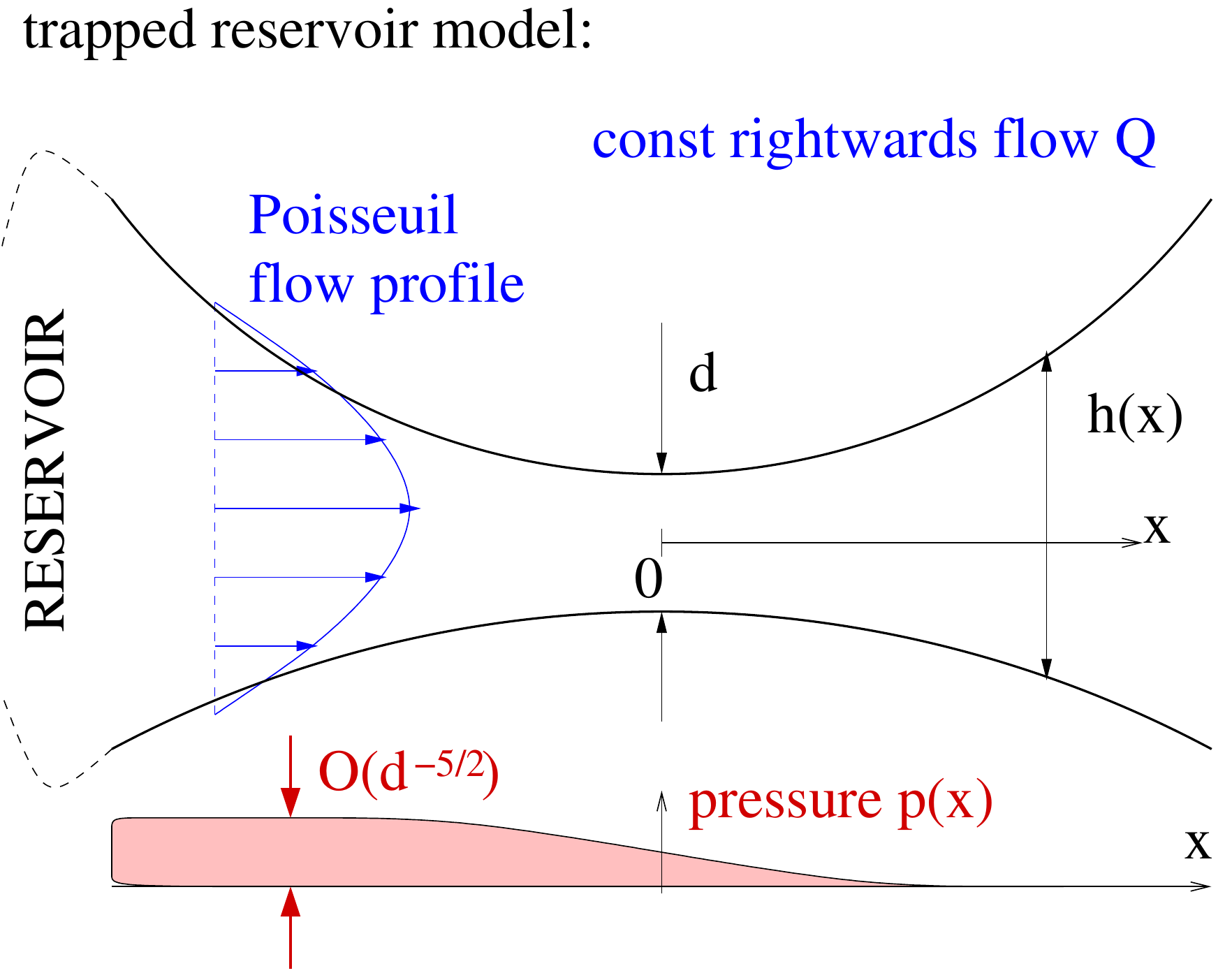}
}
\ca{Jamming with reservoir trapping (Example 4). There are two smooth ``star'' shaped particles per unit cell.
    Top row: three snapshots of the evolution, with pressure field in color and tracer points as white dots (times $t=0$, $t=0.2$, $t=0.46$).
  Bottom left: finite-time blow-up of $\mu_\tbox{eff}(t)$.
  Bottom middle: the best fit (shown on log-log scale) of $\mu_\tbox{eff}(t)$ to the power-law form \eqref{blowup}, giving $\beta \approx 2.43$.
  Bottom right: lubrication theory pipe flow asymptotic model predicting the
  power $\beta=5/2$.}{f:stars}
\efi

\bfi  % fffffffffffffffffffff
\ig{width=1.95in}{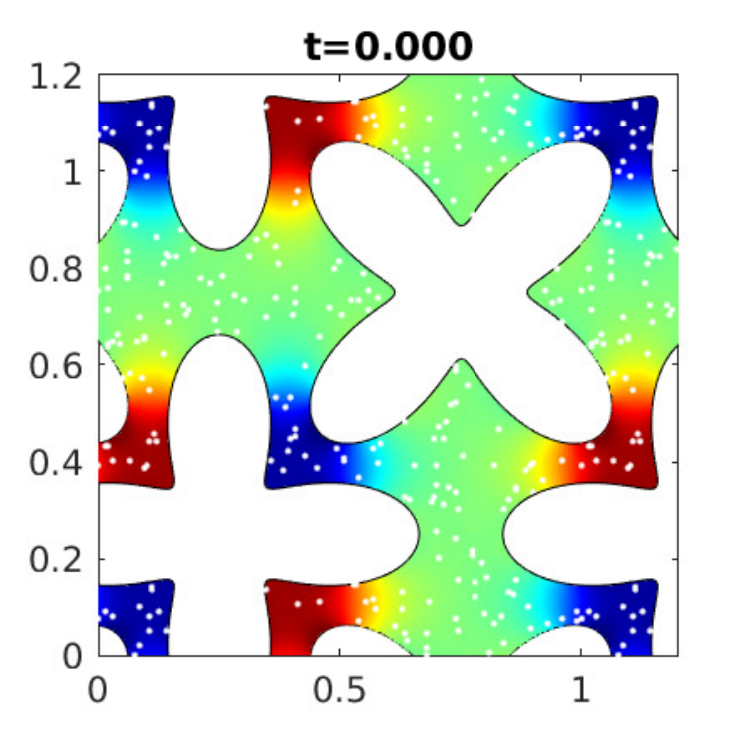}
\ig{width=1.95in}{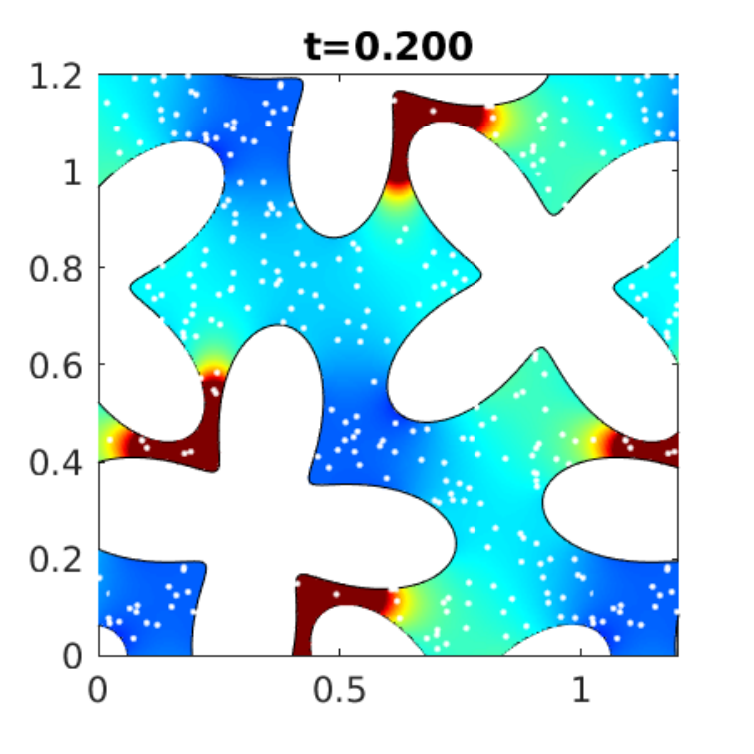}
\ig{width=1.95in}{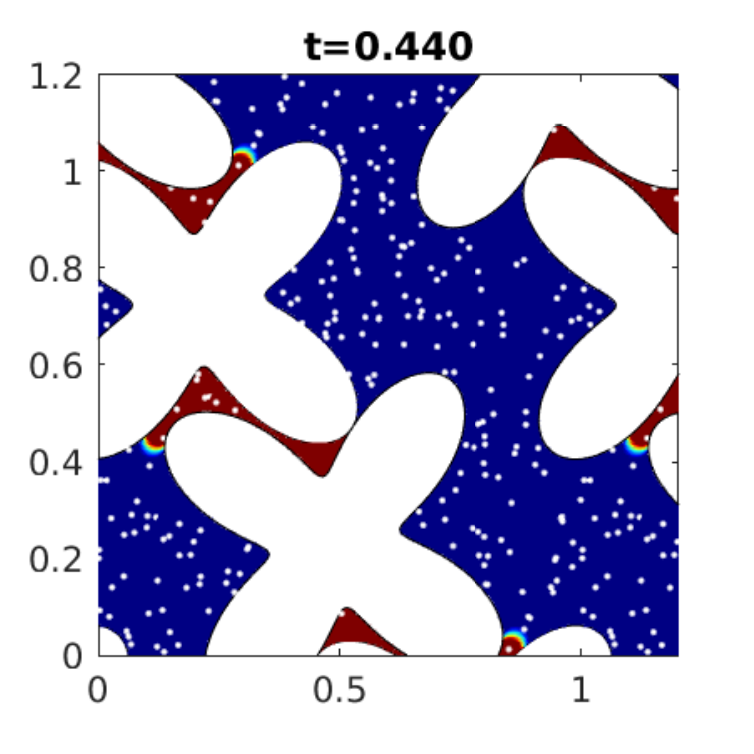}
\\
\mbox{
  \ig{width=1.8in}{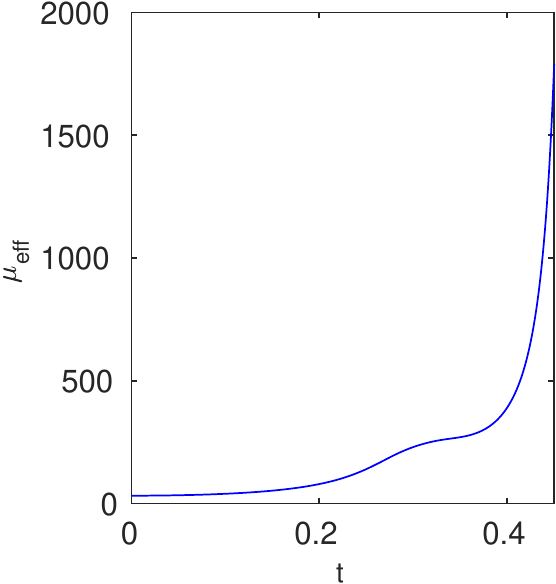}
  \;
  \ig{width=1.8in}{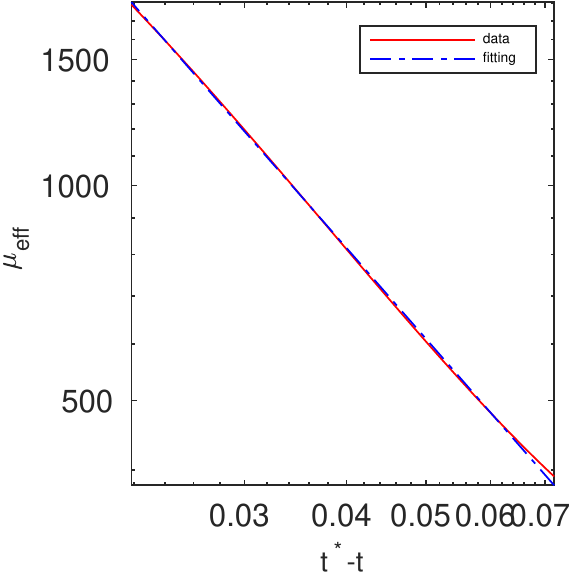}
  \;
  \ig{width=2.2in}{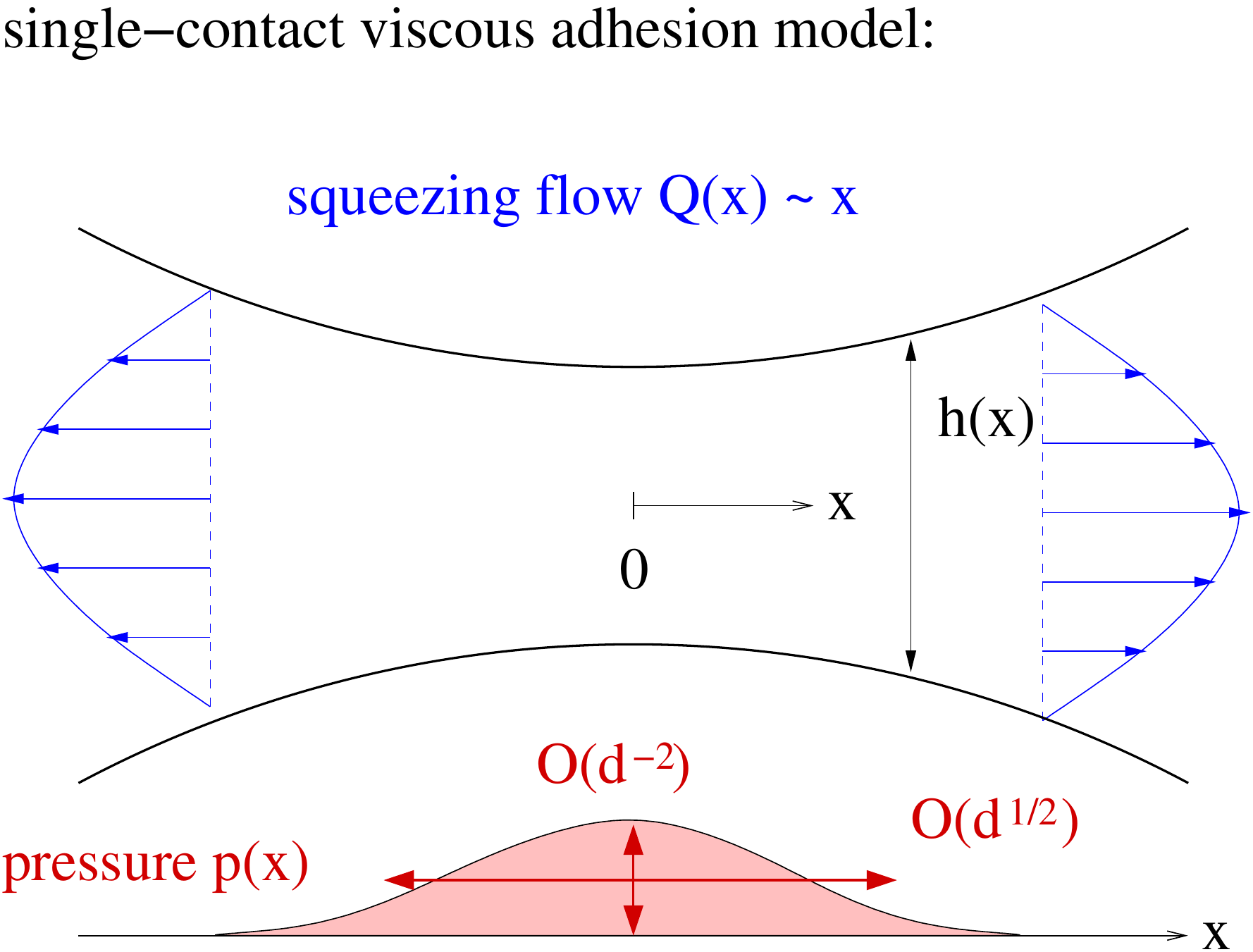}
}
\ca{Jamming with single-point contact (Example 5). There are again two smooth ``star'' shapes per unit cell, with different parameters from Fig.~\ref{f:stars}.
    Top row: three snapshots of the evolution, with pressure field in color and tracer points as white dots (times $t=0$, $t=0.2$, $t=0.44$).
  Bottom left: finite-time blow-up of $\mu_\tbox{eff}(t)$.
  Bottom middle: the best fit (shown on log-log scale) of $\mu_\tbox{eff}(t)$ to the power-law form \eqref{blowup}, giving $\beta \approx 1.3$.
  Bottom right: lubrication theory viscous adhesion asymptotic model predicting the
  power $\beta=3/2$.}{f:stars2}
\efi

\section{Conclusion}
\label{s:conc}
We have developed an integral equation based fast solver for the highly accurate simulation of a doubly-periodic suspension of arbitrary smooth rigid particles in a shearing Stokes flow.
The effort per time-step scales linearly in the complexity of the geometry.
Our method combines a general framework for periodization which is robust and flexible, a Stokes FMM, and
a special quadrature rule for close to boundary targets.
The latter gives spectral accuracy in the spatial (quasi-static) solve.
With reasonable sampling of each particle's boundary, our method proves to be accurate even for particles separated by distances as close as $10^{-4}$ times the particle size.
A first order explicit Euler method is adopted for time stepping, although our formulation easily allows higher order time stepping.
The solver models pure hydrodynamic interactions without artificial
repulsion between particles;
we have found that, at least for ellipses at moderate volume fractions,
that accurate spatial solutions and small time steps mean that
no collisions occur even after tens of shear times.
In special cases where collisions occur in finite time, we demonstrate
that a lubrication theory model correctly predicts the numerical
asymptotic power-law force blow-up.

The solver presented could find application in studying
complex nonlinear rheology, the shape and volume fraction dependence of effective viscosity, and transport and turbulence effects.
It can trivially be generalized to time-varying shear rate to
study hysteretic or finite-frequency shear oscillation effects.
It is also easy to include gravitational forces to study sedimentation.

There are several challenges that we plan to address in future work.
For guaranteed accuracy in the spatial solve at closer distances,
adaptivity (eg using panel-based discretization of the
boundaries) would be needed \cite{ojalastokes,wu2019arxiv}.
Higher-order time-stepping is essential to exploit, but it remains an issue
to handle the apparent stiffness when two or more particles come close to each other.
At higher volume fractions it appears that one cannot avoid
including artificial repulsions or collision-avoidance \cite{lu2017contact,yan2019scalable}.
The generalization to three dimensions is relatively straightforward, given
high-order surface quadratures, or schemes for spheres \cite{yan2019scalable}, and will be reported at a later date.

\section*{Acknowledgments}

We are grateful for discussions with Manas Rachh, Charlie Epstein,
Mike Shelley, and Leslie Greengard.
The Flatiron Institute is a division of the Simons Foundation.

%%%%  AAAAAAAAAAAAAAAAAAAAAAAAAAAAAAAAAAAAAAAAAAAAAAAAAAAAAAAAAAAAAAAAAAAAAA
\appendix
\section{Proof of Theorem~\ref{t:EBVP}}
\label{a:EBVP}

We are given a smooth pair $(\w,r)$ which generates the discrepancy $\g$
via \eqref{g1}--\eqref{g4}.
We wish to prove existence of $(\v,q)$ solving the EBVP
\eqref{eq:emptyBoxEq1}--\eqref{eq:emptyBoxDiscrep4}.
Let $\mbf{h} := \mu \Delta\w - \nabla r \in (C^\infty(\ccU))^2$ and
$h:=-\nabla\cdot\w \in C^\infty(\ccU)$.
Now associate $\cU$ with the flat torus $\tor$ with the same lattice vectors,
and notice that $\mbf{h}$ and $h$, as functions on $\tor$, are
generally discontinuous but bounded, hence in $L^2(\tor)$.
Then let $(\u,p)$ be a periodic solution pair to
the inhomogeneous Stokes BVP on $\tor$:
\bea
-\mu \Delta \u + \nabla p &=& \mbf{h} \qquad \mbox{ in } \tor
\label{tor1}
\\
\nabla \cdot \u &=& h \qquad \mbox{ in } \tor
\label{tor2}
~.
\eea
Recall that the Stokes system is elliptic in the Douglis--Nirenberg sense
(see, e.g., \cite[Sec.~2.2.2]{volpert}).
The manifold $\tor$ has no boundary, so ellipticity implies that the
Stokes operator in this domain is Fredholm (e.g.\ \cite[Theorem~8.53]{wloka}).
It is also self-adjoint.
Thus if $(\mbf{h},h)$ is orthogonal to the nullspace, a solution
to \eqref{tor1}--\eqref{tor2} exists.
That the nullspace is precisely the span of the constant functions,
hence is 3-dimensional, is a simple result proven in
\cite[Prop.~4.2]{ahb}.
Thus we need only show that $\mbf{h}$ and $h$ defined above integrate to zero.
We compute, using the divergence theorem for Stokes \eqref{eq:DivThmStokes},
$$
\int_\cU \mbf{h} = \int_\cU (\mu \Delta\w - \nabla r) =
\int_{\pcU} \T(\w,r) = \int_L \g_2 + \int_D \g_4 = \mbf{0}
$$
by the consistency condition \eqref{gforce} on $\g$.
Similarly, using the usual divergence theorem,
$$
\int_\cU h = -\int_\cU \nabla\cdot\w = 
-\int_{\pcU} n\cdot\w = -\int_L n\cdot\g_1 - \int_D n\cdot\g_3 = 0
$$
by \eqref{gvol}.
Thus, $(\u,p)$ exists.
Finally, let $\v = \w + \u$, and $q = r + p$.
Since the volume right-hand side terms cancel,
it is easy to check that $(\v,q)$ satisfies the EBVP
\eqref{eq:emptyBoxEq1}--\eqref{eq:emptyBoxDiscrep4}.

By ellipticity of the Stokes system,
since the driving $(\mbf{h},h)$ is $C^\infty$ in the interior $\cU$,
the same is true for the solution $(\u,p)$,
and hence, by smoothness of $(\w,r)$, also for $(\v,q)$.
\hfill $\square$

\section{Proof of Theorem~\ref{t:consistency}}
\label{a:cons}

We first show that $\g$ is consistent in the sense of \eqref{gforce}--\eqref{gvol}.
\begin{align*}
\int_L \g_2 \,ds +\int_D \g_4 \,ds &= \int_L (-T^{near}_R+T^{near}_L+\frac{\f}{2|\e_2|}) \,ds 
+\int_D (-T^{near}_U+T^{near}_D+\frac{\f}{2|\e_1|})\,ds \\
&= \frac{\f}{2|\e_2|}\int_L ds + \frac{\f}{2|\e_1|}\int_D ds +\int_{\partial\cU} T^{near} ds \\
&= \f-\f = \mbf{0}.
\end{align*}
\begin{align*}
\int_L \g_1\cdot \n \,ds +\int_D \g_3\cdot \n \,ds &= \int_L (\w_L^{near}-\w_R^{near})\cdot \n \,ds 
+\int_D (\w_D^{near}-\w_U^{near})\cdot \n \,ds \\
&= \int_{\partial\cU} \w^{near}\cdot\, ds = \int_{\partial\cU}\w\cdot\n\, ds =0
\end{align*}

We must now show that there is a $(\w,q)$ which
generates the $\g$ given in \eqref{eq:discrepCorr1}--\eqref{eq:discrepCorr4}.
Since $\y$ is in the open set $\cU$, thus a nonzero distance from $\pcU$,
there exists a mollifier $\psi\in C^\infty(\overline{\cU})$
which is identically 1 in a neighborhood of $\pcU$ but identically zero
in a neighborhood of $\y$.
Then the smooth functions $\w = \psi \w^{near}$ and $q=\psi q^{near}$
generate (in the sense of Theorem~\ref{t:EBVP}) the bracketed
terms in \eqref{eq:discrepCorr1}--\eqref{eq:discrepCorr4}.
Smooth functions may then be added to $\w$ to generate the additional
$\f$-dependent constant terms in $\g_2$ and $\g_4$.
To see this, for example, when the unit cell has height 1,
one may check using $\n_1=(0,1)$ that the smooth vector function
$\w(x_1,x_2) = (f_1/2,f_2/4)x_2^2/\mu|\e_1|$, and $q\equiv 0$,
generates $\g_4 =\f/2|\e_1|$ as needed in \eqref{eq:discrepCorr4}.
The term needed in $\g_2$ may be similarly generated by rotating a similar
function.
Thus there is a $(\w,q)$ in $C^\infty(\overline{\cU})$ generating $\g$.
Finally, we apply Theorem~\ref{t:EBVP}.
\hfill $\square$

\bibliographystyle{abbrv} %amsplain}
\bibliography{ref.bib}

\begin{thebibliography}{10}

\bibitem{klint}
L.~af~Klinteberg and A.-K. Tornberg.
\newblock Fast {E}wald summation for {S}tokesian particle suspensions.
\newblock {\em Int. J. Num. Meth. Fluids}, 76:669--698, 2014.

\bibitem{ahb}
A.~H. Barnett, G.~Marple, S.~Veerapaneni, and L.~Zhao.
\newblock A unified integral equation scheme for doubly-periodic {L}aplace and
  {S}tokes boundary value problems in two dimensions.
\newblock {\em Comm. Pure Appl. Math.}, 71(11):2334--2380, 2018.

\bibitem{bwv}
A.~H. Barnett, B.~Wu, and S.~Veerapaneni.
\newblock Spectrally-accurate quadratures for evaluation of layer potentials
  close to the boundary for the {2D} {S}tokes and {L}aplace equations.
\newblock {\em SIAM J. Sci. Comput.}, 37(4):B519--B542, 2015.

\bibitem{closeglobal}
A.~H. Barnett, B.~Wu, and S.~Veerapaneni.
\newblock Spectrally-accurate quadratures for evaluation of layer potentials
  close to the boundary for the 2d {S}tokes and {L}aplace equations.
\newblock {\em sisc}, 37(4):B519--B542, 2015.

\bibitem{bender1998reversible}
J.~Bender and N.~J. Wagner.
\newblock Reversible shear thickening in monodisperse and bidisperse colloidal
  dispersions.
\newblock {\em Journal of Rheology}, 40(5):899, 1998.

\bibitem{berlyand09}
L.~V. Berlyand, Y.~Gorb, and A.~Novikov.
\newblock Fictitious fluid approach and anomalous blow-up of the dissipation
  rate in a two-dimensional model of concentrated suspensions.
\newblock {\em Archive for Rational Mechanics and Analysis}, 193(3):585--622, 7
  2009.

\bibitem{brady1988stokesian}
J.~F. Brady and G.~Bossis.
\newblock Stokesian dynamics.
\newblock {\em Annu.\ Rev.\ Fluid Mech.}, 20(1):111--157, 1988.

\bibitem{brady1997microstructure}
J.~F. Brady and J.~F. Morris.
\newblock Microstructure of strongly sheared suspensions and its impact on
  rheology and diffusion.
\newblock {\em Journal of {F}luid {M}echanics}, 348:103--139, 1997.

\bibitem{cheng2011imaging}
X.~Cheng, J.~H. McCoy, J.~N. Israelachvili, and I.~Cohen.
\newblock Imaging the microscopic structure of shear thinning and thickening
  colloidal suspensions.
\newblock {\em Science}, 333(6047):1276--1279, 2011.

\bibitem{cwalina2016rheology}
C.~D. Cwalina, K.~J. Harrison, and N.~J. Wagner.
\newblock Rheology of cubic particles suspended in a newtonian fluid.
\newblock {\em Soft {M}atter}, 12(20):4654--4665, 2016.

\bibitem{durlofsky1987dynamic}
L.~Durlofsky, J.~F. Brady, and G.~Bossis.
\newblock Dynamic simulation of hydrodynamically interacting particles.
\newblock {\em Journal of {F}luid {M}echanics}, 180:21--49, 1987.

\bibitem{egres2005rheology}
R.~G. Egres and N.~J. Wagner.
\newblock The rheology and microstructure of acicular precipitated calcium
  carbonate colloidal suspensions through the shear thickening transition.
\newblock {\em Journal of {R}heology}, 49(3):719--746, 2005.

\bibitem{Follandbook}
G.~B. Folland.
\newblock {\em Introduction to Partial Differential Equations}.
\newblock Princeton University Press, second edition, 1996.

\bibitem{fossbrady}
D.~R. Foss and J.~F. Brady.
\newblock Structure, diffusion and rheology of {B}rownian suspensions by
  {S}tokesian {D}ynamics simulation.
\newblock {\em J.\ Fluid Mech.}, 407:167--200, 2000.

\bibitem{gadala1980shear}
F.~Gadala-Maria and A.~Acrivos.
\newblock Shear-induced structure in a concentrated suspension of solid
  spheres.
\newblock {\em Journal of {R}heology}, 24(6):799--814, 1980.

\bibitem{Krop04}
L.~Greengard and M.~C. Kropinski.
\newblock Integral equation methods for stokes flow in doubly-periodic domains.
\newblock {\em J. Eng. Math.}, 48:157--170, 2004.

\bibitem{lapFMM}
L.~Greengard and V.~Rokhlin.
\newblock A fast algorithm for particle simulations.
\newblock {\em J. Comput. Phys.}, 73:325--348, 1987.

\bibitem{guazzelli2011physical}
E.~Guazzeli and J.~F. Morris.
\newblock {\em A physical introduction to suspension dynamics}, volume~45.
\newblock Cambridge University Press, 2011.

\bibitem{gurnon2015microstructure}
A.~K. Gurnon and N.~J. Wagner.
\newblock Microstructure and rheology relationships for shear thickening
  colloidal dispersions.
\newblock {\em Journal of Fluid Mechanics}, 769:242--276, 2015.

\bibitem{haan98}
J.~J. Haan and P.~S. Steif.
\newblock Particle-phase pressure in a slow shearing flow based on the
  numerical simulation of a planar suspension of rough contacting cylinders.
\newblock {\em J. Rheology}, 42:891--916, 1998.

\bibitem{helsing_close}
J.~Helsing and R.~Ojala.
\newblock On the evaluation of layer potentials close to their sources.
\newblock {\em J. Comput. Phys.}, 227:2899--2921, 2008.

\bibitem{HW}
G.~Hsiao and W.~L. Wendland.
\newblock {\em Boundary Integral Equations}.
\newblock Applied Mathematical Sciences, Vol. 164. Springer, 2008.

\bibitem{jamali2019alternative}
S.~Jamali and J.~F. Brady.
\newblock Alternative frictional model for discontinuous shear thickening of
  dense suspensions: {H}ydrodynamics.
\newblock {\em Physical {R}eview {L}etters}, 123(13):138002, 2019.

\bibitem{karrilakim}
S.~J. Karrila and S.~Kim.
\newblock Integral equations of the second kind for stokes flow: direction
  solution for physical variables and removal of inherent accuracy limitations.
\newblock {\em Chemical engineering communications}, 82(1):123--161, 1989.

\bibitem{kelloggbook}
O.~D. Kellogg.
\newblock {\em Foundations of Potential Theory}.
\newblock Springer-Verlag Berlin Heidelberg, first edition, 1967.

\bibitem{larsonhigdon2}
R.~E. Larson and J.~J.~L. Higdon.
\newblock Microscopic flow near the surface of two-dimensional porous media.
  part 2. transverse flow.
\newblock {\em J. Fluid Mech.}, 178:119--136, 1987.

\bibitem{lu2017contact}
L.~Lu, A.~Rahimian, and D.~Zorin.
\newblock Contact-aware simulations of particulate {S}tokesian suspensions.
\newblock {\em Journal of {C}omputational {P}hysics}, 347:160--182, 2017.

\bibitem{maranzano2001effects}
B.~J. Maranzano and N.~J. Wagner.
\newblock The effects of particle size on reversible shear thickening of
  concentrated colloidal dispersions.
\newblock {\em The Journal of {C}hemical {P}hysics}, 114(23):10514--10527,
  2001.

\bibitem{melrose1995pathological}
J.~Melrose and R.~Ball.
\newblock The pathological behaviour of sheared hard spheres with hydrodynamic
  interactions.
\newblock {\em EPL (Europhysics Letters)}, 32(6):535, 1995.

\bibitem{mewis2012colloidal}
J.~Mewis and N.~J. Wagner.
\newblock {\em Colloidal suspension rheology}.
\newblock Cambridge University Press, 2012.

\bibitem{nazockdast2012microstructural}
E.~Nazockdast and J.~F. Morris.
\newblock Microstructural theory and the rheology of concentrated colloidal
  suspensions.
\newblock {\em Journal of {F}luid {M}echanics}, 713:420--452, 2012.

\bibitem{nunan84}
K.~C. Nunan and J.~B. Keller.
\newblock Effective viscosity of a periodic suspension.
\newblock {\em Journal of Fluid Mechanics}, 142:269–287, 1984.

\bibitem{ojalastokes}
R.~Ojala and A.-K. Tornberg.
\newblock An accurate integral equation method for simulating multi-phase
  {S}tokes flow.
\newblock {\em J. Comput. Phys.}, 298:145--160, 2015.

\bibitem{pantonbook}
R.~L. Panton.
\newblock {\em Incompressible Flow}.
\newblock Wiley, fourth edition, 2013.

\bibitem{rachhgreengard}
M.~Rachh and L.~Greengard.
\newblock Integral equation methods for elastance and mobility problems in two
  dimensions.
\newblock {\em SIAM J. Numer. Anal.}, 54(5):2889--2909, 2016.

\bibitem{sanganimo}
A.~S. Sangani and G.~Mo.
\newblock Inclusion of lubrication forces in dynamic simulations.
\newblock {\em Phys.\ Fluids}, 6(5):1653--1662, 1994.

\bibitem{sierou2001accelerated}
A.~Sierou and J.~F. Brady.
\newblock Accelerated stokesian dynamics simulations.
\newblock {\em Journal of {F}luid {M}echanics}, 448:115--146, 2001.

\bibitem{souzy15}
M.~Souzy, X.~Yin, E.~Villermaux, C.~Abid, and B.~Metzger.
\newblock Super-diffusion in sheared suspensions.
\newblock {\em Phys.\ Fluids}, 27:041705, 2015.

\bibitem{volpert}
V.~Volpert.
\newblock {\em Elliptic Partial Differential Equations. {V}olume 1: {F}redholm
  Theory of Elliptic Problems in Unbounded Domains}.
\newblock Birkhauser, Springer Basel, 2011.

\bibitem{wang2016spectral}
M.~Wang and J.~F. Brady.
\newblock Spectral {E}wald acceleration of stokesian dynamics for polydisperse
  suspensions.
\newblock {\em Journal of {C}omputational {P}hysics}, 306:443--477, 2016.

\bibitem{wloka}
J.~T. Wloka, B.~Rowley, and B.~Lawruk.
\newblock {\em Boundary value problems for elliptic systems}.
\newblock Cambridge University Press, 1995.

\bibitem{wu2019arxiv}
B.~Wu, H.~Zhu, A.~H. Barnett, and S.~Veerapaneni.
\newblock Solution of {S}tokes flow in complex nonsmooth {2D} geometries via a
  linear-scaling high-order adaptive integral equation scheme.
\newblock {\em preprint}, 2019.
\newblock {\tt arXiv:1909.00049 [math.NA]}.

\bibitem{xu2014microstructure}
B.~Xu and J.~F. Gilchrist.
\newblock Microstructure of sheared monosized colloidal suspensions resulting
  from hydrodynamic and electrostatic interactions.
\newblock {\em The Journal of chemical physics}, 140(20):204903, 2014.

\bibitem{yan2019scalable}
W.~Yan, E.~Corona, D.~Malhotra, S.~Veerapaneni, and M.~Shelley.
\newblock A scalable computational platform for particulate {S}tokes
  suspensions.
\newblock {\em preprint}, 2019.
\newblock {\tt arXiv:1909.06623 [cs.CE]}.

\end{thebibliography}
\end{document}